\documentclass[onefignum,onetabnum]{siamart171218}

\ifpdf
\hypersetup{
  pdftitle={An Example Article},
  pdfauthor={D. Doe, P. T. Frank, and J. E. Smith}
}
\fi

\usepackage{amsmath}
\usepackage{amssymb}
\usepackage{color}
\usepackage{epsfig}
\usepackage{tikz}
\usepackage{algorithmic}
\usepackage{mathtools}
\usepackage{geometry}
\usepackage{longtable}
\usepackage{supertabular}
\usepackage{url}
\usepackage{blkarray}
\usepackage{enumitem}
\usepackage{multirow}
\usepackage{setspace}
\usepackage{float}
\usepackage[title]{appendix}
\usepackage{empheq}
\usepackage{mathtools}

\newtheorem{assumption}{Assumption}

\DeclareMathOperator*{\argmin}{argmin}
\DeclareMathOperator*{\arginf}{arginf}

\DeclareRobustCommand{\VAN}[3]{#2}
\DeclareRobustCommand{\DE}[3]{#2}

\newcommand{\E}{\mathbb{E}}

\newcommand{\exclude}[1]{}

\newcommand{\MP}{\mathcal{P}}

\newcommand{\MN}{\mathcal{N}}
\newcommand{\MM}{\mathcal{M}}
\newcommand{\innerit}{t}
\newcommand{\outerit}{\ell}

\renewcommand{\mathfrak}{\mathcal}

\newcommand{\s}[1]{{\mathcal{#1}}^*}

\newcommand{\sigmamin}{\sigma_{\scriptsize\mbox{min}}}
\newcommand{\sigmamax}{\sigma_{\scriptsize\mbox{max}}}

\title{Adaptive Sequential SAA for Solving \\ Two-stage Stochastic Linear Programs}
\author{Raghu Pasupathy\thanks{Department of Statistics, Purdue University, West Lafayette, IN, USA, (\email{pasupath@purdue.edu}).} \and Yongjia Song\thanks{Department of Industrial Engineering, Clemson University, Clemson, SC, USA, (\email{yongjis@clemson.edu}).}}

\begin{document}

\maketitle

\begin{abstract}
We present adaptive sequential SAA (sample average approximation) algorithms to solve large-scale two-stage stochastic linear programs. The iterative algorithm framework we propose is organized into \emph{outer} and \emph{inner} iterations as follows: during each outer iteration, a sample-path problem is implicitly generated using a sample of  observations or ``scenarios," and solved only \emph{imprecisely}, to within a tolerance that is chosen \emph{adaptively}, by balancing the estimated statistical error against solution error. The solutions from prior iterations serve as \emph{warm starts} to aid efficient solution of the (piecewise linear convex) sample-path optimization problems generated on subsequent iterations. The generated scenarios can be independent and identically distributed (iid), or dependent, as in Monte Carlo generation using Latin-hypercube sampling, antithetic variates, or randomized quasi-Monte Carlo. We first characterize the almost-sure convergence (and convergence in mean) of the optimality gap and the distance of the generated stochastic iterates to the true solution set. We then characterize the corresponding iteration complexity and work complexity rates as a function of the sample size schedule, demonstrating that the best achievable work complexity rate is Monte Carlo canonical and analogous to the generic $\mathcal{O}(\epsilon^{-2})$ optimal complexity for non-smooth convex optimization. We report extensive numerical tests that indicate favorable performance, due primarily to the use of a sequential framework with an optimal sample size schedule, and the use of warm starts. The proposed algorithm can be stopped in finite-time to return a solution endowed with a probabilistic guarantee on quality.    
\end{abstract}



\section{INTRODUCTION} 
The \emph{two-stage stochastic linear program} (2SLP) is that of minimizing the real-valued function $c^\top x + \mathbb{E}[Q(x,\xi)]$ with respect to decision variables $x \in \mathbb{R}^{n_1}_+$ over a set of linear constraints $\mathcal{X} := \{x \in \mathbb{R}^{n_1}_+: Ax=b\}$, where $Q(x,\xi)$ is itself the optimal value of a random linear program (LP) parameterized by $x$. Crucially, in 2SLPs, the term $\mathbb{E}[Q(x,\xi)]$ appearing in the objective function is not observable directly; instead, $\mathbb{E}[Q(x,\xi)]$ can only be \emph{estimated} to requested precision as the sample mean $Q_n(x) := n^{-1}\sum_{i=1}^n Q(x,\xi_i)$ of optimal values $Q(x,\xi_i),i=1,2,\ldots,n$ from randomly sampled LPs. The generation of the random LPs to estimate $\mathbb{E}[Q(x,\xi)]$ is usually accomplished through Monte Carlo sampling, by generating identically distributed ``scenarios" $\xi_i, i=1,2,\ldots,n$ that may or may not be independent. 

It appears that 2SLPs were originally introduced by \cite{dantzig1955linear} and, owing to their usefulness, have been extensively studied over the last few decades~\cite{birge2011introduction}. The sample average approximation (SAA) method seems to have emerged as a popular approach to solving 2SLPs by constructing a solution estimator as follows: 
\begin{enumerate} 
\item[(i)] generate an implicit approximation of the objective function using a specified number of ``scenarios" $\xi_1,\xi_2,\ldots,\xi_n$ obtained, e.g., using Monte Carlo sampling; 
\item[(ii)] replace the 2SLP by a sample-path optimization problem~\cite{2014kimpashen,robinson1996analysis} having the objective function obtained in (i) and having the known constraint set $\mathcal{X}$, and solve it using one of a variety of decomposition approaches that have been proposed in the literature, e.g., \cite{vanAckooij_Oliveira_Song_2016,Oliveira_Sagastizabal_2014,wolf2014applying}.
\end{enumerate} 
SAA's popularity stems from its simplicity and its obvious utility within distributed settings,  where its structure lends to easy parallelization. Over the last two decades, SAA as described through (i) and (ii) has been extensively analyzed in settings that are much more general than just 2SLPs. For example, results on the consistency and rates of convergence of optimal values/solutions, large and small sample properties, and other special properties are now available through standard textbooks~\cite{shapiro2014lectures} and surveys~ \cite{2014hombay,2014kimpashen}. 

It is important to note that SAA is a paradigm and not an algorithm in that important components within the SAA framework still need to be chosen before implementation can occur. To implement the SAA paradigm as stated in (i) and (ii), a practitioner needs to select a sample size and a Monte Carlo generation mechanism in (i), and an appropriate solver/stopping-mechanism in (ii). For instance, the question of sample size choice for generating the sample-path problem in (i) has sometimes been a vexing issue, with practitioners often making this choice through trial and error, using minimum sample size bounds that have been noted to be conservative~\cite{2014kimpashen,2008lueahm,2003russha}, and more recently, using multiple sample sizes and solving multiple sample-path problems. 

A premise of this paper is that SAA's effective implementation depends crucially on the disciplined customization (to narrowly defined problem classes, e.g., 2SLPs) of choices internal to SAA. Such customization involves answering specific algorithmic questions that arise during implementation. For instance: \begin{enumerate} \item[(a)] Is it best to generate and solve (to machine precision) a single sample-path problem with a large Monte Carlo sample size or is it better to progressively and roughly solve a sequence of sample-path problems generated with increasing sample size? If the latter strategy is better, what schedule of sample sizes should be used?  \item[(b)] Recognizing that any generated sample-path problem suffers from sampling error and hence suggests not solving to machine precision, to what extent should a sample-path problem be solved? \item[(c)] What type of solvers should be used in solving the generated sample-path problems, given that the solution information to previously solved sample-path problem(s) can be fruitfully used as a \emph{warm start} to a subsequent sample-path problem? \end{enumerate}
In this paper, we rigorously investigate questions (a)--(c) for the specific case of 2SLPs. And, consistent with our earlier comments, our answers to (a)--(c) seem to be vital to attaining the encouraging numerical experience we describe in Section~\ref{sec:numerical}.  

\subsection{Summary and Insight on Main Results} 
The essence of our proposed framework is the construction of a sequential SAA framework for solving 2SLPs, where a sequence of approximate 2SLPs are generated and solved to progressively increasing precision across iterations. The framework is such that the early iterates are obtained with little computational burden since, by design, the generated sample-path problems tend to have small sample sizes and are solved imprecisely; and the later iterates can be expected to be obtained with ease as well since they tend to benefit from the warm starts using solution information obtained in previous iterations. The schedule of sample sizes and the adaptive optimality-tolerance parameters are chosen to be in lock-step, ensuring that no particular sample-path problem is ``over-solved." The framework we provide is an algorithm in the strict sense of the word in that we make specific recommendations for choosing: (i) the schedule of sample sizes to generate the sample-path problems to approximate the 2SLP, (ii) the schedule of error-tolerance parameters to which each of the generated sample-path problems is to be solved, and (iii) the solver to use when solving the sample-path problems. We also demonstrate that our framework can exploit existing results on finite-time stopping to provide solutions with probabilistic guarantees on optimality. Our extensive numerical experience on solving large-scale 2SLPs suggests that the proposed algorithm yields competitive computational performance compared with existing methods. 

We present a number of results that form the theoretical basis for the proposed algorithm. We present sufficient conditions under which the optimality gap and the distance (from the true solution set) of the algorithm's stochastic iterates converges to zero almost surely and in expectation. We also derive the corresponding iteration complexity and work complexity rates, that is, we provide upper bounds (in expectation) on the number of iterations and the number of Monte Carlo oracle calls to ensure that the solution resulting from the framework is $\epsilon$-optimal. The derived work complexity leads to an optimal sample size schedule which is shown to achieve the fastest possible convergence rate in a Monte Carlo setting.  Lastly, we demonstrate that using sample size schedules that deviate from the proposed schedule will lead to inferior convergence rates.   

We emphasize that the framework we propose is general in that it allows for the use of a wide range of dependent sampling, e.g., Latin-hypercube sampling (LHS)~\cite{mckay1979comparison}, antithetic variates~\cite{nelson1988antithetic}, and randomized quasi-Monte Carlo~\cite{glasserman2013monte,2016lec} \emph{within} a generated sample-path problem, and the reuse of scenarios \emph{across} generated sample-path problems. While we do not attempt to demonstrate that the use of such variance reduction measures is better than iid sampling, other reports~\cite{chen2015scenario,stockbridge2016variance} in the literature suggest the fruitfulness of such variance reduction techniques.

\subsection{Related Literature}

\exclude{2SLPs have been the subject of investigation for a long time~\cite{1995bir} and algorithms to solve 2SLPs can be conveniently classified based on whether or not they can treat the context $|\Xi| = \infty$. As noted in~\cite{2001zha}, the context $|\Xi| < \infty$ has generated an enormous amount of work resulting in various algorithm classes that directly exploit the finite sum structure stemming from assuming $|\Xi| < \infty$ --- see ~\cite{1995bir} and~\cite{2011chemeh} for entry points into this substantial literature. }

2SLPs have been the subject of investigation for a long time~\cite{1995bir} and algorithms to solve 2SLPs can be conveniently classified based on whether or not the probability space underlying the 2SLP is endowed with a sample space having a finite number of outcomes. As noted in~\cite{2001zha}, an enormous amount of work has been generated especially for the context where the sample space is finite, resulting in various algorithm classes that directly exploit the finite sum structure  --- see ~\cite{1995bir} and~\cite{2011chemeh} for entry points into this substantial literature. 

\exclude{When $|\Xi| = \infty$, or for that matter when $|\Xi|$ is ``very large,"} 
For 2SLPs with sample spaces having countably infinite or an uncountable number of outcomes, or for that matter even sample spaces with large cardinality, Monte Carlo sampling approaches appear to be a viable alternative~\cite{shapiro2014lectures,shapiro1998simulation,2000shahom}. In fact, sequential Monte Carlo sampling methods such as what we propose here are not new and have appeared in the stochastic programming (SP) and simulation optimization (SO) literature for several decades now~\cite{pasupathy2010choosing,2009denfer,ermoliev1992stochastic,higle1991stochastic,homem2003variable,2014hombay,shapiro1998simulation,2013wanpassch}. For instance,~\cite{ermoliev1992stochastic} proposes the stochastic quasi-gradient methods for optimization of discrete event systems,~\cite{shapiro1998simulation} suggests the idea of solving a sequence of sample-path problems with increasing sample sizes as a practical matter, and \cite{homem2003variable} gives various sufficient conditions on how fast the sample size should grow in order to ensure the consistency of the SAA estimator with varying sample sizes. For SPs where the corresponding sample-path problems are smooth optimization problems, \cite{polak2008efficient, royset2013sample} study the sample size selection problem for the sequential sampling procedure. They model the sequential sampling procedure as a stochastic adaptive control problem, by finding the optimal sample size \emph{as well as} the number of iterations that one should apply to solve the sampled problems, so that the total expected computational effort expended in the entire procedure is minimized. A surrogate model is then proposed to approximate this adaptive control model so that the sample size and the number of iterations to be employed at each iteration can be found (relatively) easily according to results from previous iterations, by solving the surrogate model. From an algorithmic perspective, the stochastic decomposition framework initially developed by~\cite{higle1991stochastic} is perhaps the most well-known practical approach that exploits the connections between statistical inference, sampling, and stochastic LPs. In addition, \cite{glynn2013simulation} proposes simulation-based Benders decomposition approach as a variant of the stochastic sub-gradient method specifically for 2SLPs and develops statistical confidence bounds for the optimal values. 

Similar to~\cite{homem2003variable}, \cite{pasupathy2010choosing,2009passch,2006pas} suggest  \emph{retrospective approximation} (RA) where a smooth stochastic optimization problem is solved through a sequence of sample-path problems generated with increasing sample sizes. Unlike in~\cite{homem2003variable}, RA methods solve the sample-path problems imprecisely, until a generally specified error-tolerance parameter is satisfied. The methods presented here can be thought to be \emph{adaptive} RA in that the  error-tolerance sequence in our current framework is adaptive since it depends explicitly on a measure of sampling variability. We find that such adaptivity is crucial for good numerical performance, although it brings additional technical difficulty due to the need to handle stopping time random variables. Also, whereas the methods in~\cite{pasupathy2010choosing,polak2008efficient, royset2013sample} do not apply to non-smooth problems such as 2SLPs, the methods we present here are tailored (through the choice of solver) to exploit the structure inherent to 2SLPs. We note in passing that adaptive sampling as a strategy to enhance efficiency of stochastic optimization algorithms has recently gained popularity --- see, for example,~\cite{bollapragada2018adaptive,bottou2018optimization,2014hasghopas,2018pasglyetal,2018sarhasetal}. 

There has also been some recent work on the question of assessing solution quality in general SPs that directly applies to the context we consider here. For example, \cite{bayraksan2011sequential,bayraksan2012fixed} propose sequential sampling methods and study conditions under which their employed optimality gap estimator is asymptotically valid in the sense of lying in a returned confidence interval with a specified probability guarantee. Applying these conditions when stipulating the sample size to be employed in each iteration, one naturally gets a highly reliable stopping criterion for the sequential sampling procedure. As we will demonstrate, the results from \cite{bayraksan2011sequential,bayraksan2012fixed} can be modified for application within a finite-time version of the proposed framework, notwithstanding the fact that the generated sample-path problems in the proposed framework need only be solved imprecisely, to within a specified error-tolerance parameter.

\subsection{Organization of the Paper}
The rest of the paper is organized as follows: Section~\ref{sec:pbstmt} presents important notation, convention, and terminology used throughout the paper, a precise problem statement of 2SLP, and a listing of key assumptions. Section~\ref{sec:ASS} introduces the proposed adaptive sequential SAA framework. Section~\ref{sec:inexact} presents various results pertaining to consistency, work complexity rates, and optimal sample size schedules. Section~\ref{sec:stopping} provides a finite stopping rule for the adaptive sequential SAA algorithm by incorporating the sequential sampling approaches proposed in \cite{bayraksan2011sequential} and~\cite{bayraksan2012fixed}. Section~\ref{sec:numerical} shows computational performance of the proposed adaptive sequential SAA framework on a variety of test instances.



\section{PROBLEM SETUP}\label{sec:pbstmt}

The 2SLP is formally stated as follows: 

\begin{align*}
 \min \ & c^\top x + q(x) \tag{$P$} \\
\text{s.t. } & x\in \mathcal{X} :=  \left\{x\in \mathbb{R}^{n_1}_+ \mid Ax = b\right\},
\end{align*}
where the $r_1 \times n_1$ matrix $A$, $r_1 \times 1$ vector $b$ and $n_1 \times 1$ vector $c$ are assumed to be fixed and known. The second-stage \emph{value function} $q(x)$ is defined as:
\begin{equation}\label{secondstagevalue}
q(x) = \mathbb{E}[Q(x,\xi)] = \int_{\Xi} Q(x,\xi) \, dP(\xi),
\end{equation} 
where for each $\xi \in \Xi$, the second-stage objective value 
\begin{equation}\label{subprob}
Q(x,\xi) = \min_{y\in \mathbb{R}^{n_2}_+} \ \left\{d(\xi)^\top y \mid W(\xi)y \geq h(\xi) - T(\xi)x\right\}.
\end{equation} 
We assume that the second-stage objective value is finite, i.e., $Q(x,\xi) > -\infty, \ \forall x\in \mathcal{X}$, and $\xi\in \Xi$. Notice that the function $q(\cdot)$ is not directly ``observable" but can be estimated pointwise by ``generating scenarios." Specifically, we assume that an iterative algorithm, during the $\ell$-th iteration, generates scenarios $\xi^{\ell}_1,\xi^{\ell}_2,\ldots,\xi^{\ell}_{m_{\ell}} \in \Xi$ that are identically distributed according to some probability measure. The resulting ``sample-path problem" due to scenarios $\xi^{\ell}_1,\xi^{\ell}_2,\ldots,\xi^{\ell}_{m_{\ell}} \in \Xi$ is given by 
\label{master-unregsubsample} 
 \begin{align*}
\min \ & c^\top x + Q^{\ell}_{m_\ell}(x) \tag{$P_{\ell}$} \\
\text{s.t. } & x\in \mathcal{X} :=  \left\{x\in \mathbb{R}^{n_1}_+ \mid Ax = b\right\},\end{align*}
where the second-stage \emph{sample-path value function} $Q^{\ell}_{m_{\ell}}(x) := m_{\ell}^{-1} \sum_{i=1}^{m_\ell} Q(x,\xi^{\ell}_i)$, and $Q(x,\xi^{\ell}_i)$ is given through \eqref{subprob}. 

To accommodate the probabilistic analysis of the \emph{adaptive iterative} algorithms we propose, we assume the existence of a filtered probability space $(\Omega, \mathfrak{F}, (\mathfrak{F}_\ell)_{\ell\geq 1},\mathbb{P})$ such that the iterates $(\hat{x}^{\ell})_{\ell \geq 1}$ generated by the algorithm we propose are adapted to $(\mathfrak{F}_\ell)_{\ell\geq 1}$. We note then that $Q^{\ell}_{m_\ell}(\cdot)$ denotes an $\mathfrak{F}_{\ell}$-measurable function estimator of $q(\cdot)$ constructed from $\xi^{\ell}_i, i=1,2,\ldots, m_\ell$ identically distributed, $\mathfrak{F}_{\ell}$-measurable random objects. The random objects $\xi^{\ell}_i, i=1,2,\ldots, m_\ell; \ell = 1,2,\ldots$ correspond to what have been called ``scenarios" in the SP literature. We will use $\xi^{\ell}$ to denote a generic $\mathcal{F}_{\ell}$-measurable outcome, and $\xi^{\ell}_1, \xi^{\ell}_2, \ldots$ to denote $\mathcal{F}_{\ell}$-measurable outcomes obtained from Monte Carlo sampling during iteration $\ell$. Thus, the problem in ($P_{\ell}$) is a ``sample-path approximation" of the problem in ($P$) and the function $Q^{\ell}_{m_\ell}(\cdot)$ is a ``sample-path approximation" of the function $q(\cdot)$. The precise sense in which the function $Q^{\ell}_{m_\ell}(\cdot)$ approximates $q(\cdot)$ will become clear when we state the standing assumptions in Section~\ref{sec:assumptions}.

The notation we use (with the superscript and subscript), while cumbersome, is needed to reflect the fact that the framework we propose allows for a variety of dependence structures of $\xi^{\ell}_i, i=1,2,\ldots, m_\ell$ within and across iterations $\ell =1,2,\ldots.$ For example, in the simplest and most prevalent case of independent and identically distributed (iid) sampling, generation is done so that the random objects $\xi^{\ell}_i,  i = 1,2,\ldots, m_\ell$ are mutually independent and identically distributed for each $\ell$; the objects $\xi^{\ell}_i, i = 1,2,\ldots, m_\ell$ can also be generated so as to satisfy chosen dependency structures that reduce variance, e.g., LHS~\cite{mckay1979comparison}, antithetic variates~\cite{2013nel}, and randomized quasi-Monte Carlo~\cite{2008leclectuf,glasserman2013monte}. Similarly, across iterations $\ell = 1,2,\ldots$, one can arrange for scenarios from previous iterations to be reused in subsequent iterations as in \emph{common random numbers}~\cite{2013nel}. Indeed, we will have to make certain assumptions on $Q^{\ell}_{m_\ell}(\cdot), \ell=1,2,\ldots$ in Section~\ref{sec:assumptions} that will implicitly impose restrictions on the nature of sampling, to ensure that $Q^{\ell}_{m_\ell}(\cdot)$ approximates $q(\cdot)$ well enough.

\subsection{Further Notation and Convention}\label{sec:notation} 

We let $\mathcal{S}^*$ denote the optimal solution set, $z^*$ the optimal value, and $\mathcal{S}^*(\epsilon) := \{x \in \mathcal{X}: c^\top x + q(x) - z^* \leq \epsilon\}$ the $\epsilon$-optimal solution set of problem $(P)$. Analogously, $\mathcal{S}_{m_{\ell}}^*$ denotes the optimal solution set, $z_{m_{\ell}}^*$ the optimal value, and $\mathcal{S}_{m_{\ell}}^*(\epsilon) := \{x \in \mathcal{X}: c^\top x + Q^{\ell}_{m_\ell}(x) - z_{m_{\ell}}^* \leq \epsilon\}$ the $\epsilon$-optimal solution set for problem $(P_{\ell})$. 

The following definitions are used extensively throughout the paper. (i) $\mathbb{R}_+$ denotes the set of non-negative real numbers.  (ii) For $x = (x_1,x_2,\ldots,x_n) \in \mathbb{R}^n$, $\|x\|_2$ refers to the Euclidean norm $\|x\|_2 = \sqrt{x_1^2 + x_2^2 + \cdots + x_n^2}$. (iii) For a real-valued continuous function $g: \mathcal{X} \to \mathbb{R}$ defined on the compact set $\mathcal{X}$, the sup-norm $\|g\|$ is defined as $\|g\|:= \max_{x \in \mathcal{X}} |g(x)|$. (iv) The distance between a point $x\in \mathbb{R}^n$ and a set $X\subseteq \mathbb{R}^n$ is defined as $\mbox{dist}(x,{X}) := \inf\{\|x-z\|_2: z\in {X}\}$, and the distance between two sets ${X},{Y}\subseteq \mathbb{R}^n$ is defined as $\mbox{dist}({X},{Y}) := \sup_{x\in {X}}\{\mbox{dist}(x,{Y})\}$. The definition we have used for $\mbox{dist}(\cdot,\cdot)$ suffices for our purposes even though it is not a metric since $\mbox{dist}({X},{Y}) \neq \mbox{dist}({Y},{X})$ in general. (v) The diameter $\mbox{diam}(X)$ of a set $X \subseteq \mathbb{R}^n$ is defined as $\mbox{diam}(X) : = \sup_{x,y \in X}\{\|x-y\|_2\}$. (vi) The projection of a point $x\in \mathbb{R}^n$ onto a set ${X}\subseteq \mathbb{R}^n$ is defined as $\mbox{proj}\left(x,{X}\right) := \arg\inf_{z\in {X}}\{\|x-z\|_2\}$. (vii) $|{X}|$ denotes the cardinality of set ${X}$. (viii) For a sequence of $\mathbb{R}^d$-valued random variables $\{Z_n\}, Z$, we say $Z_n \to Z \mbox{ a.s.}$ to mean that $\{Z_n\}$
converges to $Z$ almost surely, that is, with probability one. We say that $Z_n$ converges to $Z$ in $L^2$-norm if $\mathbb{E}[\|Z_n\|_2] \to \mathbb{E}[\|Z\|_2]$ as $n \to \infty$. (See~\cite{1995bil} for modes of convergence of sequences of random variables.) 

\subsection{Assumptions}\label{sec:assumptions}

The following is a list of assumptions that we will use to prove various results in the paper. Assumption~\ref{ass:boundedfeasible} and Assumption~\ref{ass:varbd} are standing assumptions in that we will assume these to hold always. Assumption~\ref{ass:fungrowth} will be invoked as and when needed. 


\begin{assumption}[Condition on Relatively Complete Recourse]\label{ass:boundedfeasible} The first-stage feasible region $\mathcal{X}$ of problem $(P)$ is compact; furthermore, Problem $(P)$ has relatively complete recourse, that is, 
\[
\mathbb{P}\left\{\left(y\in \mathbb{R}^{n_2}_+ : W(\xi)\,y \geq h(\xi) - T(\xi)\,x  \right) = \emptyset \right\} = 0, \ \forall x \in \mathcal{X}.
\]
\end{assumption} 

\begin{assumption} [Condition on Estimator Quality] \label{ass:varbd} The individual observations comprising the Monte Carlo estimator have finite variance, that is, for all $\ell \geq 1$, \begin{equation}\label{finvar} \sup_{x \in \mathcal{X}} \mathbb{V}\mbox{ar}(Q(x,\xi^{\ell}) \, \vert \, \mathcal{F}_{\ell-1}) < \infty \emph{\mbox{ a.s.}} \end{equation} Moreover, the Monte Carlo estimator error decays at the canonical Monte Carlo rate, that is, there exists a constant $\kappa_0 < \infty$ such that for all $\ell \geq 1$,  \begin{equation}\label{msebd} \mathbb{E}\left[\|\bar{\epsilon}_m\|^2\, \vert \, \mathcal{F}_{\ell -1} \right] \leq \frac{\kappa_0}{m} \emph{\mbox{ a.s.}},\end{equation} where the sample-mean error function $\bar{\epsilon}_m(x) : = Q_m^{\ell}(x) - q(x) = m^{-1} \sum_{j=1}^m (Q(x,\xi_j^{\ell}) - q(x)).$ (The $\|\cdot\|$ appearing in \eqref{msebd} is the sup-norm defined in Section~\ref{sec:notation}).
\end{assumption}

\begin{assumption}[Condition on Growth Rate of Objective Function]\label{ass:fungrowth} The (true) objective function exhibits $\gamma_0$-first-order growth on $\mathcal{X}$, that is, $$\gamma_0 := \sup_{s}\{s: c^\top x + q(x) - z^* \geq s \,\emph{\mbox{dist}}(x,\s{S}) \quad \forall x \in \mathcal{X}\} > 0.$$ 

\end{assumption}


Some form of regularity such as \eqref{finvar} in Assumption~\ref{ass:varbd} is routinely made in the SP literature~\cite{bayraksan2006assessing} and is generally easy to satisfy in 2SLPs when the feasible region $\mathcal{X}$ is compact. 

The condition (\ref{msebd}) in Assumption~\ref{ass:varbd} has been stated for generality, to subsume many contexts that involve dependent and biased sampling, and needs justification. To get a clear sense of the conditions under which \eqref{msebd} in Assumption~\ref{ass:varbd} holds, let's first observe that in the \emph{iid unbiased} context, that is, when $\xi_j^{\ell}, j=1,2,\ldots$ are iid and $\mathbb{E}[Q(x,\xi_j^{\ell}) - q(x) \, \vert \, \mathcal{F}_{\ell-1}] = 0 \mbox{ a.s.}$, the vast body of recent literature on concentration inequalities~\cite{2007mas,2013boulugmas,2011ledtal,1994tal} guarantees that (\ref{msebd}) holds under a variety of moment conditions on $Q(x,\xi^{\ell})$. For a general result that can be directly applied in the iid unbiased context, see~\cite[Proposition 3.1]{2010dumetal} established for Banach spaces. (Much of the literature on concentration inequalities is focused on sharp quantifications of the tail probabilities associated with $\bar{\epsilon}_m$, and thus characterize the constant $\kappa_0$ indirectly; our proposed algorithms do not rely on knowing $\kappa_0$.)

In the \emph{dependent but unbiased} sampling context, that is, when $\mathbb{E}[Q(x,\xi_j^{\ell}) - q(x) \, \vert \, \mathcal{F}_{\ell-1}] = 0 \mbox{ a.s.}$ but $\xi_j^{\ell}, j=1,2,\ldots$ are not necessarily independent, Assumption~\ref{ass:varbd} holds in many popular settings where the estimator can be written as an alternate sum of iid unbiased random variables at each $x \in \mathcal{X}$. For instance, consider using \emph{antithetic variates}~\cite{2013nel}, where for even $m$ we set $\xi_j^{\ell} := U_j ^{\ell}\in (0,1), \xi_{j+1}^{\ell} = 1-U_j^{\ell}, j=1,3,5,\ldots,m-1$. Then, $Q_m^{\ell}(x)$ can be written as the sample mean of $m/2$ (ignoring non-integrality) iid unbiased random variables, each of which is the sum of the two dependent random variables $Q(x,U_j^{\ell})$ and $Q(x,1-U_j^{\ell})$, implying that Assumption~\ref{ass:varbd} again holds. Similarly, if one chooses stratified sampling~\cite{glasserman2013monte} as a variance reduction technique, then $Q_m^{\ell}(x), x \in \mathcal{X}$ can be written as a finite convex combination of sample means, each of which is composed of iid random variables that are unbiased with respect to the conditional means.   

Assumption~\ref{ass:varbd} can be shown to hold in other dependent sampling settings such as LHS~\cite{mckay1979comparison} as well. To see this, we ``construct" a $d$-dimensional random variable $\xi^{\ell}_j := (\xi^{\ell}_{1j}, \xi^{\ell}_{2j}, \ldots,\xi^{\ell}_{dj}) \in [0,1)^d$ where $\xi^{\ell}_{ij} = m^{-1}(\pi_{ij} + U_{ij})$, $\pi_i = (\pi_{i1}, \pi_{i2}, \ldots, \pi_{im}), i=1,2,\ldots,d$ is each a uniform random permutation of $(0,1,2,\ldots,m-1)$, $U_{ij} \sim [0,1)$, and $U_{ij}$'s and $\pi_i$'s are independent. Under this setup, we see that $\xi^{\ell}_{ij} \sim U[0,1)$, $\xi^{\ell}_i \sim U[0,1)^d$, and that $Q_{m_{\ell}}^{\ell}(x)$ is an unbiased estimator of $q(x)$ that is constructed from dependent random variables. Furthermore, under this setup, and as shown in~\cite[p. 245]{mckay1979comparison} and~\cite[Section 10.3]{2013owe}, $\mbox{Var}(Q(x,\xi)) < \infty$ guarantees that $\mbox{Var}(Q^{\ell}_{m_{\ell}}(x)) = \sigma_0^2/m_{\ell} + o(m_{\ell}^{-1}) = O(m_{\ell}^{-1}),$ where $\sigma_0^2 = \mathbb{E}[(Q(x,\xi) - Q^{\scriptsize{\mbox{add}}}(x,\xi))^2]$ and $Q^{\scriptsize{\mbox{add}}}(x,\xi))$ is the additive approximation of $Q(x,\xi)$ obtained using ANOVA.  See also~\cite{1987ste} for large sample properties in the LHS context.

Randomized quasi-Monte Carlo (RQMC) is a broad class of variance reduction  methods that subsumes various dependent sampling techniques, and where arguments similar to what we have outlined for LHS apply when considering the variance of the estimator $Q^{\ell}_{m_{\ell}}$. See~\cite[Section 2]{2018lec}, and the specific RQMC methods listed there, to see how RQMC yields estimators having variance at least as small as what is obtained using naive Monte Carlo, thus guaranteeing $O(m_{\ell}^{-1})$ variance.

We recognize that we have limited all of the above discussion on dependent sampling by fixing $x \in \mathcal{X}$. A complete treatment of Assumption~\ref{ass:varbd} that involves dependence across $x \in \mathcal{X}$ will require us to consider the behavior of the random function $Q_{m_{\ell}}^{\ell}(\cdot)$ by directly making assumptions on the vector $(d(\xi), W(\xi), h(\xi),T(\xi))$ appearing in the second-stage problem~\eqref{subprob}. In general, some sort of a stipulation on the quality of the Monte Carlo estimator is needed to provide reasonable guarantees relating to convergence and convergence rates. For example, in Chapter 5 of~\cite{shapiro2014lectures}, we see that even for convergence of sample-path optimal values of SAA to the true optimal value, one needs uniform convergence (across $x \in \mathcal{X}$) of the sample-path functions. 

Finally, Assumption \ref{ass:fungrowth} is a standard regularity condition~\cite{shapiro2014lectures} having to do with the growth behavior of the true objective function. Specifically, Assumption~\ref{ass:fungrowth} imposes a  minimum growth condition on the true objective function $c^\top x + q(x)$.

\section{ADAPTIVE SEQUENTIAL SAA}\label{sec:ASS}
In this section, we present the proposed adaptive sequential SAA algorithm. The proposed algorithm is based on the following three high-level ideas. \begin{enumerate}\item[(1)] Instead of solving (to any given precision) a single sample-path problem that is generated with a large pre-specified sample size, solve (using a chosen Solver-$\mathcal{A}$) a \emph{sequence} of sample-path problems generated with increasing sample sizes according to a sample size schedule. \item[(2)] Use the solution information obtained from solving each sample-path problem as a \emph{warm start} for solving
the subsequent sample-path problem. \item[(3)] To ensure that no particular sample-path problem is over-solved, solve each generated sample-path problem only \emph{imprecisely} to within an optimality tolerance parameter that is adaptively chosen by explicitly considering the inherent sampling error resulting from the choice of
sample size.\end{enumerate} 

\begin{algorithm}
\caption{An adaptive sequential SAA framework.} 
\label{algo:adaptive-sequential-nonterminating}
\begin{algorithmic}[1]
\STATE{{\bf Input:} Solver-$\mathcal{A}$, a sampling scheme, constants $\nu, \sigma_{\scriptsize\mbox{min}}, \sigma_{\scriptsize\mbox{max}} \in (0,\infty)$, with $\sigma_{\scriptsize\mbox{min}} < \sigma_{\scriptsize\mbox{max}}$. \STATE Set $\ell\leftarrow 0$.}
 \FOR{$\outerit = 1,2,\cdots$}
        \STATE{Select sample size $m_{\ell}$ for outer iteration $\ell$ and draw a sample $\MM_{\ell} := \{\xi^{\ell}_1,\xi^{\ell}_2,\ldots,\xi^{\ell}_{m_{\ell}}\}$.}
 	\FOR{$\innerit = 1,2,\cdots$} 	
 		\STATE{Using Solver-$\mathcal{A}$ on $(P_{\ell})$, execute $\innerit$-th inner iteration. \STATE Obtain candidate solution $\hat{x}^{\ell,t}$, gap estimate $G^{\ell,\innerit}$ and variance parameter estimate $\hat{\sigma}_{\ell,\innerit}$.}
		\IF{$G^{\ell,\innerit} \leq \epsilon_{\ell,t} := \nu \, m_{\ell}^{-1/2}\, \mbox{proj}\left(\hat{\sigma}_{\ell,t},[\sigma_{\scriptsize\mbox{min}},\sigma_{\scriptsize\mbox{max}}]\right)$}
 		
 			\STATE{Break the inner loop with a candidate solution $\hat{x}^\ell := \hat{x}^{\ell,t}$.}
 		\ENDIF
 	\ENDFOR
 	\STATE{Set $\ell \leftarrow \ell+1$.}
\ENDFOR
\end{algorithmic}
\end{algorithm}

As can be seen through the listing for Algorithm~\ref{algo:adaptive-sequential-nonterminating}, the iterative framework maintains outer iterations that are indexed by $\ell$, each of which is composed of inner iterations indexed by $t$. During the $\ell$-th outer iteration, the $\ell$-th sample-path problem $(P_{\ell})$ with sample $\MM_\ell := \{\xi^{\ell}_1,\xi^{\ell}_2,\ldots,\xi^{\ell}_{m_{\ell}}\}$ is generated and solved inexactly up to precision $\epsilon_{\ell}$ using an iterative optimization algorithm (generically called Solver-$\mathcal{A}$) for non-smooth convex programs, e.g., the subgradient method~\cite{nedic2001convergence}, level bundle method~\cite{lemarechal1995new}. We will see later that any solver that satisfies a certain imposition on convergence rate can be used as Solver-$\mathcal{A}$. The iterations of Solver-$\mathcal{A}$ thus constitute the \emph{inner iterations} generating a sequence of inner solutions $\hat{x}^{\ell,\innerit}, t = 1,2,\ldots$

During each inner iteration $t$, an upper bound estimate $G^{\ell,\innerit}$ of the optimality gap associated with $\hat{x}^{\ell,\innerit}$ is readily available for any variant of cutting plane algorithms, where a lower approximation $\check{Q}^{\ell,t}_{m_\ell}(\cdot)$ to $Q^\ell_{m_\ell}(\cdot)$ is maintained and iteratively updated. Specifically, the objective value corresponding to $\hat{x}^{\ell,\innerit}$, $\bar{z}^\ell_t := c^\top \hat{x}^{\ell,\innerit} + Q^\ell_{m_\ell}(\hat{x}^{\ell,\innerit})$, gives an upper bound for $z^*_{m_\ell}$.  The true optimality gap associated with $\hat{x}^{\ell,\innerit}$, $\bar{z}^\ell_t - z^*_{m_\ell}$, can then be overestimated if a lower bound $\underline{z}^\ell_t$ for $z^*_{m_\ell}$ is provided. Such a lower bound $\underline{z}^\ell_t$ can be obtained, e.g., by solving $\underline{z}^\ell_t = \min_{x\in \mathcal{X}}\left\{c^\top x + \check{Q}^{\ell,t}_{m_\ell}(x)\right\}$. This optimality gap estimate, $G^{\ell, \innerit} := \bar{z}^\ell_t - \underline{z}^\ell_t$, is then compared against an estimate of the sampling error of the true solution of the $\ell$-th sample-path problem calculated using $\hat{x}^{\ell,\innerit}$.  Precisely, the inner iterations terminate when 
\begin{equation}\label{innerterm}
G^{\ell,\innerit} < \epsilon_{\ell,t} := \nu \, m_{\ell}^{-1/2}\, \mbox{proj}\left(\hat{\sigma}_{\ell,t},[\sigmamin,\sigmamax]\right),
\end{equation}
where $\sigmamin, \sigmamax, \nu > 0$ are chosen constant parameters, and, as usual, the sample variance
\begin{equation}\label{sehat1}
\hat{\sigma}^2_{\ell,\innerit}:= \frac{1}{m_{\ell}} \sum_{i=1}^{m_{\ell}}\left[Q(\hat{x}^{\ell,\innerit},\xi^{\ell}_i)-Q^\ell_{m_\ell}(\hat{x}^{\ell,\innerit})\right]^2.
\end{equation} We informally call $\epsilon_{\ell,t}$ appearing in (\ref{innerterm}) the \emph{error tolerance}; notice that the condition in (\ref{innerterm}) is meant to keep the estimate of the solution error (as measured by the optimality gap $G^{\ell,\innerit}$) in balance with the sampling error, as  measured by the error tolerance $\epsilon_{\ell,t}$. The constants $\sigmamin, \sigmamax$ appearing in \eqref{innerterm} have been introduced for practical purposes only, to hedge against the rare event that we generate scenarios resulting in an extremely large or extremely small value of the sample variance. Thus:
\begin{enumerate}
\item[--] if $G^{\ell,\innerit} \geq \epsilon_{\ell,t}$, that is, the upper bound estimate of the optimality gap for solving the current sample-path problem is no less than a factor of the sampling error estimate, continue to the next \emph{inner iteration} $\innerit+1$;
\item[--] otherwise, stop solving the current sample-path problem, that is, terminate the inner iterations, define $\epsilon_\ell := \epsilon_{\ell, t}$, obtain a new scenario set $\MM_{\ell+1} := \{\xi^{\ell+1}_1,\xi^{\ell+1}_2,\ldots,\xi^{\ell+1}_{m_{\ell+1}}\}$ with sample size $m_{\ell+1}$ and continue to the next \emph{outer iteration} $\ell+1$.
\end{enumerate} When the inner termination condition (\ref{innerterm}) is achieved, we stop the inner iterations, record the solution $\hat{x}^{\ell,\innerit}$ at termination as the current candidate solution $\hat{x}^\ell$, obtain a new scenario set $\MM_{\ell+1}$ and start a new outer iteration $\ell+1$ with $\hat{x}^\ell$ as the initial candidate solution. Additional information such as the optimal dual multipliers collected up to outer iteration $\ell$ can also be used to warm start the outer iteration $\ell+1$. The process is then repeated until a stopping criterion for the outer iteration of Algorithm~\ref{algo:adaptive-sequential-nonterminating} is satisfied by the candidate solution $\hat{x}^\ell$. We defer our specification of the outer stopping criterion to Section~\ref{sec:stopping}.

Algorithm~\ref{algo:adaptive-sequential-nonterminating} is \emph{adaptive} in that $\epsilon_\ell$ is not pre-specified --- it is a function of scenarios $\MM_{\ell} := \{\xi^{\ell}_1,\xi^{\ell}_2,\ldots,\xi^{\ell}_{m_{\ell}}\}$ used in the $\ell$-th outer iteration. Adaptivity is crucial for practical efficiency and when incorporated in our way, avoids several mathematical complexities that otherwise manifest.

We end this section with a result that quantifies the quality of estimators used within Algorthm~\ref{algo:adaptive-sequential-nonterminating}. Specifically, Theorem~\ref{thm:mcestimator} quantifies the quality of $Q^{\ell}_{m_{\ell}}(\cdot)$ as an estimator of $q(\cdot)$.  

\begin{theorem}[Monte Carlo Estimator Quality] \label{thm:mcestimator} Suppose Assumption~\ref{ass:boundedfeasible} and Assumption~\ref{ass:varbd} hold, and the sequence of sample sizes $(m_{\ell})_{\ell \geq 1}$ is chosen so that the following condition holds: 
\begin{equation}\label{sampsizeloginc} 
\sum_{\ell =1}^{\infty} \frac{1}{\sqrt{m_{\ell}}} < \infty, \,\, m_{\ell} \geq 1. \tag{SS-A}
\end{equation}  
Then $\sup_{x \in \mathcal{X}} |Q^{\ell}_{m_{\ell}}(x) - q(x)| = 0 \emph{\mbox{ a.s.}} \emph{\mbox{ as }} \ell \to \infty.$ 
\end{theorem}

\begin{proof} We can write for $t>0$, a.s.,
\begin{align}\label{concjen}
\mathbb{P}\left\{\sup_{x \in \mathcal{X}} |Q^{\ell}_{m_{\ell}}(x) - q(x)|>t \, \vert \, \mathcal{F}_{\ell-1} \right\} & \leq t^{-1}\mathbb{E}\left[\sup_{x \in \mathcal{X}} |Q^{\ell}_{m_{\ell}}(x) - q(x)| \, \vert \, \mathcal{F}_{\ell-1} \right] \nonumber \\ 
&= t^{-1}\,\mathbb{E}\left[\left(\sup_{x \in \mathcal{X}} |Q^{\ell}_{m_{\ell}}(x) - q(x)|^2 \right)^{1/2}\, \vert \, \mathcal{F}_{\ell-1} \right] \nonumber \\
& \leq t^{-1}\,\left(\mathbb{E}\left[(\sup_{x \in \mathcal{X}} |Q^{\ell}_{m_{\ell}}(x) - q(x)|^2 \, \vert \, \mathcal{F}_{\ell-1} \right]\right)^{1/2}  \leq \frac{\sqrt{\kappa_0}\,t^{-1}}{\sqrt{m_{\ell}}},
\end{align} 
where the first line in (\ref{concjen}) is from Markov's inequality~\cite{1995bil}, the third from (concave) Jensen's inequality~\cite{1995bil}, and the last from Assumption~\ref{ass:varbd}. Conclude from \eqref{concjen}, the assumed bound \eqref{sampsizeloginc}, and the filtered version of the Borel-Cantelli lemma~\cite{1991wil}, that the assertion of the theorem holds.


\end{proof}

We note that the condition in \eqref{sampsizeloginc} is weak --- any sequence $(m_{\ell})_{\ell \geq 1}$ that satisfies $m_{\ell} \geq \ell^{2+\epsilon}$ for large enough $\ell$ and some $\epsilon>0$ satisfies the condition in \eqref{sampsizeloginc}. The utility of Theorem~\ref{thm:mcestimator} is that it connects uniform almost sure convergence of the Monte Carlo estimator with the moment assumption specified through Assumption~\ref{ass:varbd}. See~\cite{homem2003variable} for analogous results for pointwise convergence. We are now ready to undertake the consistency of the iterates $(\hat{x}^{\ell})_{\ell \geq 1}$ generated by Algorithm~\ref{algo:adaptive-sequential-nonterminating}. 

\section{CONSISTENCY}\label{sec:inexact}

In this section, we treat the consistency of the stochastic iterates generated by the proposed algorithm. By consistency, we mean convergence guarantees (both almost sure and in expectation) associated with the true function values at the stochastic iterates, and the stochastic iterates themselves. This section also sets up the foundation for work complexity results of the subsequent section.

We begin with Lemma~\ref{lem:prox} --- a result on the behavior of approximate minimizers of a sequence of convex functions that uniformly converge to a limit function. We emphasize that this result is stated in a \emph{deterministic} setting and will become very useful in explaining the behavior of the sample paths in the stochastic context in the subsequent section. It also appears to be interesting in its own right due to applicability in the context of optimization with a deterministic inexact oracle. See~\cite{1983erm,1987pol,1978pol} for more on such problems. A complete proof is provided in the appendix of the online supplementary document~\cite{online}.

\begin{lemma}\label{lem:prox} Let $(f_k)_{k\geq 1}, f_k: \mathcal{X} \subset \mathbb{R}^d \to \mathbb{R}$ be a sequence of real-valued convex functions defined on the compact set $\mathcal{X}$. Let $f:\mathcal{X} \to \mathbb{R}$ be a real-valued function such that $f_k$ uniformly converges to $f$, that is, $$\lim_{k \to \infty} \sup_{x \in \mathcal{X}} |f_k(x) - f(x)| =0.$$ Denote $\delta_{k+1} := \underset{{x \in \mathcal{X}}}{\sup}\, |f_k(x) - f_{k+1}(x)|$, $\mathcal{S}^*_f := \underset{x \in \mathcal{X}}{\arg\min}\{f(x)\}$ and $v^* := \underset{x \in \mathcal{X}}{\min}\{f(x)\}$. The point $x_k$ is said to be $\epsilon_k$-optimal to $f_{k}$ over $\mathcal{X}$ if $x_k$ satisfies $ |f_k(x_k) - v_k^*| \leq \epsilon_k,$ where $v_k^* := \underset{x \in \mathcal{X}}{\min}\{f_k(x)\}$. Suppose the sequences $(\delta_k)_{k\geq 1}$, $(\epsilon_k)_{k \geq 1}$ satisfy \begin{equation} \label{errrate1} \sum_{j=1}^{\infty} \delta_j < \infty; \quad \sum_{j=1}^{\infty} \epsilon_j < \infty. \tag{SS-1} \end{equation} 
Then the following assertions hold.  \begin{enumerate} \item[\emph{(a)}] $f(x_k) \to v^*$ as $k \to \infty$; \item[\emph{(b)}] for each $k \geq 1$, $f(x_k) - v^* \leq 2\sum_{j=k}^{\infty} \delta_j + 2 \sum_{j=k}^{\infty} \epsilon_j.$ \end{enumerate} If $f$ obeys a growth rate condition, that is, there exist $\tau>0, \gamma>0$ such that for all $x \in \mathcal{X}$, \begin{equation}\label{growthrate} f(x) - v^* \geq \tau \, \emph{\mbox{dist}}(x,\mathcal{S}^*_f)^{\gamma}, \ \text{then}\end{equation}  \begin{enumerate} \item[\emph{(c)}] for each $k \geq 1$, $\emph{\mbox{dist}}(x_k,\mathcal{S}^*_f) \leq \left( 2 \tau^{-1} \left( \sum_{j=k}^{\infty} \delta_j + \sum_{j=k}^{\infty} \epsilon_j \right)\right)^{\frac{1}{\gamma}}.$ \end{enumerate} 
\end{lemma}

We emphasize that the postulates of Lemma~\ref{lem:prox} allow $f_k,f$ to be non-smooth convex functions without a unique minimizer. Moreover, Lemma~\ref{lem:prox} guarantees through assertion (a)  that the function values at the iterates converge to the optimal value $v^*$ at a rate characterized in assertion (b). A corresponding rate guarantee on the distance between the $k$-th approximate solution $x_k$ and the true solution set $\mathcal{S}_f^*$ can be given under a growth rate assumption on the objective function $f$. 

Notice that Lemma~\ref{lem:prox} \emph{does not} assert that the sequence of approximate solutions $(x_k)_{k \geq 1}$ converges to a point in the solution set $\mathcal{S}_f^*$, but only that the distance between the sequence $(x_k)_{k\geq 1}$ and the set $\mathcal{S}^*_f$ converges to zero. A guarantee such as convergence to a point is not possible as is, but may be possible by solving regularized versions of $f_k$, assuming the regularization parameters are chosen appropriately. This question lies outside the scope of the current paper.  
  
We are now ready to characterize consistency in the stochastic context. The first (Theorem~\ref{lem:distbasicineq}) of these results asserts that the true function values at the iterates generated by the proposed algorithm converge to the optimal value almost surely and in expectation. Furthermore, if the objective function $q(\cdot)$ satisfies a growth condition on $\mathcal{X}$, then similar guarantees can be provided on the distance between the solutions $(\hat{x}^{\ell})_{\ell \geq 1}$ and the solution set $\mathcal{S}^*$.

\begin{theorem}[Consistency] \label{lem:distbasicineq} Suppose Assumption~\ref{ass:boundedfeasible} and~\ref{ass:varbd}, and the sample size condition \eqref{sampsizeloginc} hold, the following assertions about the iterates $(\hat{x}^{\ell})_{\ell \geq 1}$ generated by Algorithm~\ref{algo:adaptive-sequential-nonterminating} are true.  \begin{enumerate} \item[\emph{(a)}] $c^T\hat{x}^{\ell} + q(\hat{x}^{\ell}) \to z^* \emph{\mbox{ a.s.}}$ as $\ell \to \infty$; \item[\emph{(b)}] $\mathbb{E}\left[c^T\hat{x}^{\ell}+ q(\hat{x}^{\ell})\right] \to z^*$ as $\ell \to \infty$.
\end{enumerate} If Assumption~\ref{ass:fungrowth} also holds, then the following assertions hold as well. \begin{enumerate} \item[\emph{(c)}] $\emph{\mbox{dist}}(\hat{x}^{\ell},\mathcal{S}^*) \to 0 \emph{\mbox{ a.s.}}$ as $\ell \to \infty$; \item[\emph{(d)}] $\mathbb{E}[\emph{\mbox{dist}}(\hat{x}^{\ell},\mathcal{S}^*)] \to 0$ as $\ell \to \infty$.
\end{enumerate}
\end{theorem}

\begin{proof} We will prove assertion (a) by demonstrating that the postulates for Lemma~\ref{lem:prox} (a) are satisfied except on a set (of sample-paths) of measure zero. 

We know that $\hat{x}^{\ell}$ is $\epsilon_{\ell}$-optimal to problem $(P_{\ell})$, that is, $|c^T\hat{x}^{\ell} + Q^{\ell}_{m_{\ell}}(\hat{x}^{\ell}) - z^*_{m_{\ell}}| \leq \epsilon_{\ell}.$ We also know that $Q^{\ell}_{m_{\ell}}(\cdot)$ is convex on $\mathcal{X}$, and from Theorem~\ref{thm:mcestimator}, $Q^{\ell}_{m_{\ell}}(\cdot)$ is uniformly convergent to $q(\cdot)$. In preparation to invoke Lemma~\ref{lem:prox}, denote $\delta_{\ell+1} := \sup_{x \in \mathcal{X}} |Q^{\ell}_{m_{\ell+1}}(x) - Q^{\ell}_{m_{\ell}}(x)|$ and notice that \begin{align}\label{deltasplitnoreg} \delta_{\ell} \leq \sup_{x \in \mathcal{X}} |Q^{\ell+1}_{m_{\ell+1}}(x) - q(x)| + \sup_{x \in \mathcal{X}}|Q^{\ell}_{m_{\ell}}(x) - q(x)| \nonumber := \zeta^{\ell+1}_{m_{\ell+1}} + \zeta^{\ell}_{m_{\ell}}. \end{align} The inequality in (\ref{deltasplitnoreg}) and Assumption~\ref{ass:varbd} imply that \begin{align}\label{deltajenbdnoreg} \mathbb{E}\left[\sum_{\ell=1}^n \delta_{\ell}\right] = \sum_{\ell=1}^n \mathbb{E}\left[\delta_{\ell}\right] & \leq \sum_{\ell=1}^n \mathbb{E}\left[\left(\zeta^{\ell+1}_{m_{\ell+1}}+\zeta^{\ell}_{m_{\ell}}\right)\right] \leq \sqrt{\kappa_0}\left(\sum_{\ell=1}^n \frac{1}{\sqrt{m_{\ell+1}}} + \sum_{\ell=1}^n \frac{1}{\sqrt{m_{\ell}}}\right),\end{align} where the last inequality in \eqref{deltajenbdnoreg} follows from Jensen's inequality~\cite[Theorem 5.1.3]{2010dur} applied to Assumption~\ref{ass:varbd}. Thus, 
\begin{equation}\label{expsumrtbdnoreg}\mathbb{E}\left[\sum_{\ell=1}^{\infty}  \delta_{\ell}\right] =  \sum_{\ell=1}^{\infty} \mathbb{E}\left[\delta_{\ell}\right] \leq \sqrt{\kappa_0}\left(\sum_{\ell=1}^{\infty} \frac{1}{\sqrt{m_{\ell+1}}} + \sum_{\ell=1}^{\infty} \frac{1}{\sqrt{m_{\ell}}}\right),
\end{equation} 
where the equality is due to the monotone convergence theorem~\cite[Theorem 16.2]{1995bil} and the inequality is due to \eqref{deltajenbdnoreg}. The inequality in \eqref{expsumrtbdnoreg} together with the sample size condition (\ref{sampsizeloginc}) implies that $\mathbb{E}\left[\sum_{\ell=1}^{\infty} \delta_{\ell}\right] < \infty$, and hence that $\sum_{\ell=1}^{\infty} \delta_{\ell} < \infty \mbox{ a.s.}$ Also, recall that the error tolerance sequence $(\epsilon_{\ell})_{\ell \geq 1}$ in Algorithm~\ref{algo:adaptive-sequential-nonterminating} has been chosen as $\epsilon_{\ell} = \nu\frac{1}{\sqrt{m_{\ell}}}\mbox{proj}(\hat{\sigma}_{\ell},[\sigmamin,\sigmamax]).$ This choice implies that $\sum_{\ell=1}^{\infty} \epsilon_{\ell} < \infty \mbox{ a.s.}$ The two inequalities above imply that all postulates leading to assertions (a) and (b) in Lemma~\ref{lem:prox} are satisfied on a set (of sample-paths) of measure one; we thus conclude that the assertion (a) of the theorem holds. The assertion in (b) follows from the assertion in (a) since the function $q(\cdot)$ is continuous on the compact set $\mathcal{X}$ and is hence bounded. 

If Assumption~\ref{ass:fungrowth} is satisfied, we know that \begin{equation}\label{distbdstoch} \mbox{dist}(\hat{x}^{\ell},\mathcal{S}^*) \leq \gamma_0^{-1}\left(c^T \hat{x}^{\ell} + q(\hat{x}^{\ell}) - z^*\right).\end{equation} Use assertion (a) and \eqref{distbdstoch} to conclude that assertion (c) holds. Furthermore, since $\mathcal{X}$ is compact, $\mbox{dist}(\hat{x}^{\ell},\mathcal{S}^*)$ is bounded and hence assertion (d) holds as well.

\end{proof}

Theorem~\ref{lem:distbasicineq} gives strong guarantees on the consistency of the objective function value at the iterates generated by Algorithm~\ref{algo:adaptive-sequential-nonterminating}. However, as is implied by assertion (c) of Theorem~\ref{lem:distbasicineq}, the solutions $(\hat{x}^{\ell})_{\ell \geq 1}$ can be guaranteed to only ``converge into" the true solution set $\mathcal{S}^*$ in the sense that the distance between $\hat{x}^{\ell}$ and the set $\mathcal{S}^*$ converges to zero almost surely and in expectation, and not that the sequence $(\hat{x}^{\ell})_{\ell \geq 1}$ is guaranteed to converge to a point. We are now ready to treat convergence rates in the stochastic context.

\section{ITERATION AND WORK COMPLEXITY GUARANTEES}\label{sec:complexity}
Theorem~\ref{lem:distbasicineq} guarantees that the sequence of iterates $(\hat{x}^\ell)_{\ell \geq 1}$ generated by Algorithm~\ref{algo:adaptive-sequential-nonterminating} are such that the corresponding objective function values converge to the optimal value almost surely and in expectation, and the iterates converge ``into" the true solution set $\mathcal{S}^*$, that is, their distance from $\mathcal{S}^*$ converges to zero almost surely and in expectation. In this section, we will provide a rigorous sense of how fast such convergence happens. Specifically, we provide complexity results that characterize the rate at which the optimality gap and the distance (from $\mathcal{S}^*$) converge to zero as a function of the iteration number and the total workload incurred through a specific iteration.   

The first result characterizes the sample-path \emph{iteration complexity} of the proposed algorithm, that is, the rate at which the convergence (as specified through Theorem~\ref{lem:distbasicineq}) happens as a function of iteration $\ell$. 

\begin{theorem}[Iteration Complexity] \label{thm:itercomp} Suppose that Assumption~\ref{ass:boundedfeasible} and Assumption~\ref{ass:varbd} hold, and that the sample size sequence obeys the following geometric increase for $\ell \geq 1$: \begin{equation}\label{sampsizegeoinc} m_{\ell} = c_1 \, m_{\ell-1}, \quad c_1 \in (1,\infty). \tag{SS-C}\end{equation} Then, \begin{equation} \label{gapiterrate} \mathbb{E}\left[ c^T\hat{x}^{\ell} + q(\hat{x}^{\ell}) - z^*\right] \leq 2 \,\kappa_2 \,c_1^{-\ell/2}, \text{ where } \kappa_2 :=  \sqrt{\frac{c_1}{m_1}}\,\frac{1}{\sqrt{c_1}-1}\left(\sqrt{\kappa_0}(\sqrt{c_1}+1) + \sigma_{\scriptsize{\emph{\mbox{max}}}}\sqrt{c_1\nu}\right). \end{equation}

If Assumption~\ref{ass:fungrowth} holds as well, then \begin{equation} \label{distiterrate} \mathbb{E}\left[ \emph{\mbox{dist}}(\hat{x}^{\ell},\mathcal{S}^*)\right] \leq 2 \,\tau_0^{-1} \kappa_2 \,c_1^{-\ell/2}.\end{equation}
\end{theorem} \begin{proof} Recall $\delta_{\ell+1} := \sup_{x \in \mathcal{X}} |Q^{\ell}_{m_{\ell+1}}(x) - Q^{\ell}_{m_{\ell}}(x)|$ and that the error tolerance sequence $(\epsilon_{\ell})_{\ell \geq 1}$ in Algorithm~\ref{algo:adaptive-sequential-nonterminating} has been chosen as $\epsilon_{\ell} = \nu\frac{1}{\sqrt{m_{\ell}}}\mbox{proj}(\hat{\sigma}_{\ell},[\sigmamin,\sigmamax]).$ From arguments in the proof of Theorem~\ref{lem:distbasicineq}, we know that $\sum_{\ell=1}^{\infty} \delta_{\ell} < \infty$ a.s., and that $\sum_{\ell=1}^{\infty} \epsilon_{\ell} < \infty$ a.s. This means that we can invoke assertion (b) of Lemma~\ref{lem:prox} on a set of measure one, that is, we have for each $\ell \geq 1$, \begin{equation}\label{optgapbd} c^\top\hat{x}^{\ell} + q(\hat{x}^{\ell}) - z^* \leq 2\sum_{k=\ell}^{\infty} \delta_{\ell} + 2\sum_{k=\ell}^{\infty} \epsilon_{\ell} \quad
 \mbox{ a.s. }\end{equation} From the monotone convergence theorem~\cite[Theorem 16.2]{1995bil}, Assumption~\ref{ass:varbd}, and the sample size choice \eqref{sampsizegeoinc}, we see that 
\begin{align}\label{sumdeltaagain} 
\mathbb{E}\left[ \sum_{k=\ell}^{\infty} \delta_k \right] & = \sum_{k=\ell}^{\infty} \mathbb{E}\left[  \delta_k \right] \leq \sqrt{\kappa_0}\left( \sum_{k = \ell}^{\infty} \frac{1}{\sqrt{m_{k+1}}} + \sum_{k = \ell}^{\infty} \frac{1}{\sqrt{m_{k}}} \right) \leq c_1^{-\ell/2}\,\sqrt{\frac{\kappa_0c_1}{m_1}}\,\frac{\sqrt{c_1}+1}{\sqrt{c_1}-1}.
\end{align} 
Also, since $\epsilon_{\ell} = \nu\frac{1}{\sqrt{m_{\ell}}}\mbox{proj}(\hat{\sigma}_{\ell},[\sigmamin,\sigmamax])$, we see that 
\begin{equation}\label{condgapbdinitial} 
\mathbb{E}\left[\epsilon_{\ell}\right] \leq \sqrt{\frac{\nu}{m_{\ell}} \sigmamax^2},
\end{equation} 
and hence 
\begin{align} \label{epsilonbdinitial} \mathbb{E}\left[\sum_{k=\ell}^{\infty} \epsilon_{k} \right] &:= \lim_{n \to \infty} \mathbb{E}\left[\sum_{k=\ell}^{n} \epsilon_k \right] = \lim_{n \to \infty} \sum_{k=\ell}^{n} \mathbb{E}\left[ \epsilon_k \right] \leq c_1^{-\ell/2}\,\sqrt{\frac{\nu \,\sigmamax^2c_1}{m_1}}\frac{\sqrt{c_1}}{\sqrt{c_1}-1},  
\end{align} 
where the inequality in \eqref{epsilonbdinitial} is due to \eqref{condgapbdinitial} and the sample size choice \eqref{sampsizegeoinc}. From \eqref{sumdeltaagain}, \eqref{epsilonbdinitial} and \eqref{optgapbd}, we conclude that the first assertion of the theorem (appearing in \eqref{gapiterrate}) holds. The second assertion of the theorem (appearing in \eqref{distiterrate}) follows trivially from the growth condition and the first assertion.

\end{proof}

Iteration complexity results such as that in Theorem~\ref{thm:itercomp} are generally of limited value (especially by themselves) in sampling contexts because they characterize the convergence rate in terms of the \emph{iteration number}, which is not reflective of the total computational work done. A more useful characterization of the convergence rate is what has been called \emph{work complexity}, which is essentially the error (in function value or distance from solution set) expressed as a function of the total computational work done, which for the current context includes the total number of second stage LPs solved. We take up this question next.

Towards characterizing the work complexity of the proposed algorithm, recall the iterative process: during iteration $\ell$, a chosen solver that we generically call Solver-$\mathcal{A}$ uses the solution $\hat{x}^{\ell-1}$ from the previous iteration as well as the dual vector information collected so far (for the special case of fixed recourse~\cite{vanAckooij_Oliveira_Song_2016,higle1991stochastic}) as ``warm start," and solves the sample-path problem $(P_{\ell})$ generated with sample $\mathcal{M}_{\ell} := \{\xi^{\ell}_1,\xi^{\ell}_2,\ldots,\xi^{\ell}_{m_{\ell}}\}$ to within tolerance $\epsilon_{\ell}$, that is, find $\hat{x}^{\ell} \in \mathcal{S}_{m_{\ell}}^*(\epsilon_{\ell})$. Given this structure, it makes sense then that the rapidity with which a point $\hat{x}^{\ell}$ is identified will play a central role in determining the overall work complexity of the proposed algorithm. Accordingly, we now make an assumption on the nature of Solver-$\mathcal{A}$ being used to solve the sample-path problem $(P_{\ell})$.

\begin{assumption}\label{ass:solver}
The Solver-$\mathcal{A}$ executed on the problem \emph{($P_{\ell}$)} having a piecewise linear convex objective, and with an initial solution $\hat{x}^{\ell-1} \in \mathcal{X}$, exhibits iteration complexity $\Lambda^2_{\ell}\,\emph{\mbox{dist}}^2\left(\hat{x}^{\ell-1}, S_{m_{\ell}}^*\right)\epsilon^{-2}$ to obtain an $\epsilon$-optimal solution, that is, 
\begin{equation}\label{ass:solvercomplexity}
\left(c^T\hat{x}^{\ell,t} + Q_{m_{\ell}}(\hat{x}^{\ell,t})\right) - z^*_{m_{\ell}}\leq \Lambda_{\ell} \frac{  \emph{\mbox{dist}}\left(\hat{x}^{\ell-1}, S_{m_{\ell}}^*\right)}{\sqrt{t}}, \quad t = 1,2,\ldots,
\end{equation}
where $\hat{x}^{\ell,t}$ is the $t$-th iterate returned by Solver-$\mathcal{A}$, and $S^*_{m_\ell}$ is the set of optimal solutions corresponding to problem \emph{($P_{\ell}$)}. Denote the growth-rate $\gamma_{\ell}$ of the sample-path function \begin{equation}\label{sampgrowth}\gamma_{\ell} : = \sup_{s}\left\{s: c^Tx + Q_{m_{\ell}}^{\ell}(x) - z_{m_\ell}^* \geq s \, \emph{\mbox{dist}}(x,S_{m_{\ell}}^*)  \quad \forall x \in \mathcal{X}\right\},\end{equation} there exists $\lambda < \infty$ such that \begin{equation}\label{constvar}\mathbb{E}\left[\left(\frac{\Lambda_{\ell}}{\gamma_{\ell}}\right)^2 \, \vert \, \mathcal{F}_{\ell-1} \right] \leq \lambda^2 < \infty \mbox{ a.s.}\end{equation} 
\end{assumption} Assumption~\ref{ass:solver} has been stated in a way that preserves generality of our theory, with the intent of allowing any choice of Solver-$\mathcal{A}$ as long as the stipulation of Assumption~\ref{ass:solver} is met. Furthermore, we emphasize that Assumption~\ref{ass:solver} has been stated for piecewise linear convex objectives, since the objective function of the sample-path problem $(P_{\ell})$ is piecewise linear convex. For instance, a number of well-known subgradient algorithms provide a guaranteed iteration complexity of the sort stipulated in~\eqref{ass:solvercomplexity} of Assumption~\ref{ass:solver} even for convex non-smooth objectives. For example, the standard subgradient descent algorithm having the iterative structure $x_{t+1} = x_t - \alpha_{t} \partial h(x_{t}), t=0,1,2,\ldots$ for solving the convex optimization problem $\min_{x\in \mathcal{X}}\{h(x)\}$, when executed with constant step size $\alpha_{t} = \epsilon/M^2$ and $\|\partial h(x)\|\leq M, \ \forall x\in \mathcal{X}$, satisfies the complexity requirement stated in Assumption~\ref{ass:solver}. Another recent example is a variant of the level bundle method~\cite{Cruz_Oliveira_2014} under an idealized assumption. In our numerical experiments presented in Section~\ref{sec:numerical}, we use an implementable variant of the level bundle method as Solver-$\mathcal{A}$, which is described in greater detail in the appendix of the online supplementary document~\cite{online}

The assumption appearing in~\eqref{constvar} on the finiteness of the second moment of the ratio $\Lambda_{\ell}/\gamma_{\ell}$ is a stipulation on the extent of the ``ill-conditioning" of the sample-path problems. To see this, consider  using the level method~\cite[Chapter 3]{2004nes} as Solver-$\mathcal{A}$ in the proposed algorithm. It follows from a well-known result~\cite[pp. 163]{2004nes} that $\Lambda_{\ell}$ then satisfies \begin{equation}\label{follows}\Lambda_{\ell} \leq \frac{M_{\ell}}{\sqrt{\alpha (1-\alpha)^2 (2 - \alpha)}}, \quad \alpha \in (0,1)\end{equation} where $\alpha \in (0,1)$ is a user-chosen constant within the level method, and $M_{\ell} := \sup_{x \in \mathcal{X}}\{ \| c + \partial Q_{m_{\ell}}(x) \| \}$ is the supremum norm (taken over the fixed compact set $\mathcal{X}$) of the sub-gradient associated with the sample-path function. It follows from~\eqref{follows} then that \begin{equation}\label{followsagain} \frac{\Lambda_{\ell}}{\gamma_{\ell}} \leq \frac{1}{\sqrt{\alpha(1-\alpha)^{2}(2-\alpha)}}\,\frac{M_{\ell}}{\gamma_{\ell}},\end{equation} where the ratio $M_{\ell}/\gamma_{\ell}$ has the interpretation of the ``condition number" of the $\ell$-th sample-path problem. It is in this sense that the condition appearing in~\eqref{constvar} can be violated in pathological settings where, persistently, the sample-path function remains ``steep" in certain directions but ``flat" in others. Also, notice that from the Cauchy-Schwarz inequality, the condition in~\eqref{constvar} is satisfied, e.g., if the fourth moments of $\Lambda_{\ell}$ and $\gamma_{\ell}^{-1}$ exist, i.e., $\mathbb{E}[\Lambda_{\ell}^4 \, \vert \, \mathcal{F}_{\ell-1}] < \infty$ and $\mathbb{E}[\gamma_{\ell}^{-4} \, \vert \, \mathcal{F}_{\ell-1}] < \infty$ a.s. The following lemma is an obvious consequence of Assumption~\ref{ass:solver}.  

\begin{lemma}\label{level-complexity} 
Suppose Assumption~\ref{ass:boundedfeasible} and~\ref{ass:solver} hold. Let $N_{\ell}$ denote the number of iterations by Solver-$\mathcal{A}$ to solve problem $(P_{\ell})$ to within optimality gap $\epsilon_{\ell}>0$ starting at $\hat{x}^{\ell-1}$, i.e., $N_{\ell} := \inf\left\{\bar{t}: \left(c^\top \hat{x}^{\ell,t} + Q^\ell_{m_\ell}(\hat{x}^{\ell,t})\right) - z^*_{m_{\ell}} \leq \epsilon_{\ell} \mbox{ for all } t \geq \bar{t}, \quad \hat{x}^{\ell,0}:= \hat{x}^{\ell-1}\right\}.$ Then, $\exists \Lambda_{\ell} \in \mathcal{F}_{\ell}$:  
\begin{equation*}
\mathbb{P}\left\{N_{\ell} > \Lambda^2_{\ell} \frac{\left(\emph{\mbox{dist}}(\hat{x}^{\ell-1},\mathcal{S}^*_{m_{\ell}})\right)^2}{\epsilon_{\ell}^2} \, \vert \, \mathcal{F}_{\ell-1} \right\} = 0 \emph{\mbox{, and }} \mathbb{E}\left[\left(\frac{\Lambda_{\ell}}{\gamma_{\ell}}\right)^2 \, \vert \, \mathcal{F}_{\ell-1} \right] < \infty \emph{\mbox{ a.s.}}
\end{equation*}  
\end{lemma} 

We will now combine the iteration complexities characterized in Theorem~\ref{thm:itercomp} and Lemma~\ref{level-complexity} to characterize the work complexity of the proposed Algorithm~\ref{algo:adaptive-sequential-nonterminating}.

\begin{theorem}\label{geometric} 
Suppose Assumption~\ref{ass:boundedfeasible},~\ref{ass:varbd} and~\ref{ass:solver} hold. Define $W_L := \sum_{\ell = 1}^L \tilde{W}_\ell$, where $\tilde{W}_\ell$ is the number of second-stage LPs solved during the $\ell$-th outer iteration of Algorithm~\ref{algo:adaptive-sequential-nonterminating}. Suppose $(m_{\ell})_{\ell \geq 1}$ satisfies the geometric increase sampling condition in (\ref{sampsizegeoinc}).  Then, for $L \geq 1$, \begin{equation}\label{workcomplexity}\mathbb{E}\left[\left(c^T\hat{x}^L + q(\hat{x}^{L}) - z^*\right)\right] \leq \tau_0/\mathbb{E}[\sqrt{W_L}], \text{ where $\tau_0$ is a constant independent of $L$.}\end{equation} If Assumption~\ref{ass:fungrowth} also holds, then for $L \geq 1$,
\begin{equation}\label{workdistcomplexity}
\mathbb{E}\left[\emph{\mbox{dist}}(\hat{x}^L,\mathcal{S}^*)\right] \leq \tau_0 \gamma_0^{-1} / \mathbb{E}[\sqrt{W_L}], \text{ where $\gamma_0$ is the growth-rate constant in Assumption~\ref{ass:fungrowth}.}
\end{equation} 
\end{theorem}

\begin{proof}
According to Lemma~\ref{level-complexity}, and recalling that up to $m_\ell$ second-stage LPs are solved in each iteration (e.g., when one employs a scenario decomposition algorithm), for every $\ell \geq 1$:
\begin{align}\label{workbdprelim} 
\mathbb{E}\left[\sqrt{\tilde{W}_\ell}\right]  = \mathbb{E}\left[ \sqrt{N_{\ell}\, m_{\ell}}\right]  \leq \mathbb{E}\left[\mathbb{E}\left[\Lambda_{\ell}\frac{\left(\mbox{dist}(\hat{x}^{\ell-1},\mathcal{S}^*_{m_{\ell}})\right)}{\epsilon_\ell} \, \sqrt{m_\ell} \, \bigg\vert \, \mathcal{F}_{\ell -1} \right]\right].
\end{align} 
Using \eqref{workbdprelim} and $\epsilon_{\ell} := \nu\,m_{\ell}^{-1/2}\mbox{proj}(\hat{\sigma}_{\ell},[\sigmamin,\sigmamax])$, we get for large enough $\ell$ that
\begin{align}\label{prelimsplit}
\mathbb{E}\left[\sqrt{\tilde{W}_\ell}\right]  &\leq \frac{m_{\ell}}{\nu\sigmamin}\mathbb{E}\left[\mathbb{E}\left[\Lambda_{\ell}\,\,\mbox{dist}(\hat{x}^{\ell-1},\mathcal{S}^*_{m_{\ell}}) \, \vert \, \mathcal{F}_{\ell-1}\right] \right] \nonumber \\
& \leq \frac{m_{\ell}}{\nu\sigmamin}\mathbb{E}\left[\left(\mathbb{E}\left[\left(\frac{\Lambda_{\ell}}{\gamma_{\ell}} \right)^{2}\,\vert \, \mathcal{F}_{\ell-1} \right]\right)^{\frac{1}{2}}\left(\mathbb{E}\left[\left( c^T\hat{x}^{\ell-1} + Q^{\ell}_{m_{\ell}}(\hat{x}^{\ell-1}) - z^*_{m_{\ell}}\right)^2\,\vert \, \mathcal{F}_{\ell-1} \right]\right)^{\frac{1}{2}}\right] \nonumber \\
& \leq \frac{m_{\ell}\lambda}{\nu\sigmamin}\mathbb{E}\left[\left(\mathbb{E}\left[\left( c^T\hat{x}^{\ell-1} + Q^{\ell}_{m_{\ell}}(\hat{x}^{\ell-1}) - z^*_{m_{\ell}}\right)^2 \, \vert \, \mathcal{F}_{\ell-1}\right]\right)^{\frac{1}{2}}\right] \nonumber \\
& \leq \frac{m_{\ell}\lambda}{\nu\sigmamin}\left(\mathbb{E}\left(\mathbb{E}\left[\left( c^T\hat{x}^{\ell-1} + Q^{\ell}_{m_{\ell}}(\hat{x}^{\ell-1}) - z^*_{m_{\ell}}\right)^2 \, \vert \, \mathcal{F}_{\ell-1}\right]\right)\right)^{\frac{1}{2}} \nonumber \\
 &\leq \frac{m_{\ell} \lambda}{\nu\sigmamin}\left(\mathbb{E}\left[ \left(\delta_{\ell} + \epsilon_{\ell-1} + |z^*_{m_{\ell-1}} - z^*_{m_{\ell}}|\right)^2\right]\right)^{\frac{1}{2}}, 
 \end{align} 
where the second inequality in~\eqref{prelimsplit} uses the Cauchy-Schwarz inequality (conditionally) and the definition in~\eqref{sampgrowth} of the sample-path growth rate, the third inequality uses the finite second moment assumption in~\eqref{constvar} of Assumption~\ref{ass:solver}, the fourth inequality uses the concavity of the square root function, and the last inequality uses $\delta_{\ell} := \sup_{x \in \mathcal{X}}\{|Q_{m_{\ell}}^{\ell}(x) - Q_{m_{\ell-1}}^{\ell-1}(x)|\}$ and the fact that $\hat{x}^{\ell-1} \in \mathcal{F}_{\ell-1}$ is $\epsilon_{\ell-1}$-optimal to $c^Tx + Q^{\ell-1}_{m_{\ell-1}}(x)$ over the set $\mathcal{X}$. Next, let $x^*_{\ell} \in S^*_{m_{\ell}}$, $x^*_{\ell-1} \in S^*_{m_{\ell-1}}$ and observe that \begin{align}\label{mindiff}|z^*_{m_{\ell-1}} - z^*_{m_{\ell}}| & \leq (c^Tx^*_{\ell} + Q^{\ell-1}_{m_{\ell-1}}(x^*_{\ell}) - z^*_{m_{\ell}}) + (c^Tx^*_{\ell-1} + Q^{\ell}_{m_{\ell}}(x^*_{\ell-1}) - z^*_{m_{\ell-1}}) \nonumber \\ &\leq 2\sup_{x \in \mathcal{X}}\left\{|Q^{\ell}_{m_{\ell}}(x) - Q^{\ell-1}_{m_{\ell-1}}(x)|\right\} =  2\delta_{\ell}.\end{align} Using~\eqref{mindiff} in~\eqref{prelimsplit}, we get for large enough $\ell$,
\begin{equation}\label{split}
\mathbb{E}\left[\sqrt{\tilde{W}_{\ell}}\right] \leq \lambda \frac{m_{\ell}}{\nu\sigmamin}\left(\mathbb{E}\left[\left(3\delta_{\ell} + \epsilon_{\ell-1}  \right)^2 \right]\right)^{\frac{1}{2}} \leq \lambda \frac{m_{\ell}}{\nu\sigmamin}\left(\mathbb{E}\left[18\delta_{\ell}^2 + 2\nu^2\frac{\sigmamax^2}{m_{\ell-1}}  \right]\right)^{\frac{1}{2}}, 
\end{equation} 
where the second inequality above uses $(a+b)^2 \leq 2a^2 + 2b^2$ and $\epsilon_{\ell} := \nu \, m_{\ell}^{-1/2}\mbox{proj}(\hat{\sigma}_{\ell},[\sigmamin,\sigmamax])$. Observing that $W_L = \sum_{\ell=1}^L \tilde{W}_{\ell}$, \eqref{split} implies that \begin{align}\label{worksep} \mathbb{E}[\sqrt{W_L}] & \leq \sum_{\ell=1}^L \mathbb{E}\left[\sqrt{\tilde{W}_{\ell}}\right] \leq \sum_{\ell=1}^L \lambda \frac{m_{\ell}}{\nu\sigmamin}\left(18\mathbb{E}\left[\delta_{\ell}^2\right] + 2\nu^2\frac{\sigmamax^2}{m_{\ell-1}}  \right)^{\frac{1}{2}} \nonumber \\ & \leq \sum_{\ell=1}^L \lambda  \frac{m_{\ell}}{\nu\sigmamin}\left(18\frac{\kappa_0(1+c_1+2\sqrt{c_1})}{m_{\ell}} + 2\nu^2c_1\frac{\sigmamax^2}{m_{\ell}}  \right)^{\frac{1}{2}} \nonumber \\ &\leq \frac{\lambda}{\nu\sigmamin}\left(18\kappa_0(1+c_1+2\sqrt{c_1}) + 2\nu^2c_1\sigmamax^2 \right)^{\frac{1}{2}} \sum_{\ell=1}^L \sqrt{m_{\ell}} \nonumber \\ &= \frac{\lambda}{\nu\sigmamin}\left(18\kappa_0(1+c_1+2\sqrt{c_1}) + 2\nu^2c_1\sigmamax^2 \right)^{\frac{1}{2}} \frac{\sqrt{m_1}}{\sqrt{c_1}-1}\left(c_1^{L/2}-1\right),\end{align} where the third inequality follows since $\mathbb{E}[\delta^2_{\ell}] \leq \kappa_0\left(m_{\ell}^{-1/2} + m_{\ell-1}^{-1/2}\right)^2$ holds from Assumption~\ref{ass:varbd}, and from further algebra (also see from the proof of Theorem~\ref{lem:distbasicineq}). Also, we know from \eqref{optgapbd} that for each $L \geq 1$, $c^T\hat{x}^L + q(\hat{x}^{L}) - z^* \leq 2\sum_{\ell=L}^{\infty} \left(\delta_\ell + \epsilon_\ell\right) \mbox{ a.s.}$, and hence, for $L \geq 1$, \begin{align}\label{optgapbdexp} \mathbb{E}\left[\left(c^T\hat{x}^L + q(\hat{x}^{L}) - z^*\right)^2\right] & \leq 4\mathbb{E}\left[\left(\lim_{n \to \infty}\sum_{\ell=L}^{n} \left(\delta_\ell + \epsilon_\ell\right)\right)^2\right] = 4\lim_{n \to \infty}\mathbb{E}\left[\left(\sum_{\ell=L}^{n} \left(\delta_\ell + \epsilon_\ell\right)\right)^2\right] \nonumber \\ & \leq 4\sum_{\ell=L}^{\infty} \mathbb{E}\left[(\delta_{\ell} + \epsilon_{\ell})^2 \right] + 8 \sum_{\ell = L}^{\infty} \left(\mathbb{E}\left[(\delta_{\ell} + \epsilon_{\ell})^2 \right]\right)^{1/2} \sum_{j = \ell+1}^{\infty} \left(\mathbb{E}\left[(\delta_{j} + \epsilon_{j})^2 \right]\right)^{1/2},\end{align} where the equality is from the monotone convergence theorem~\cite[Theorem 16.2]{1995bil}, and the last inequality follows from the repeated application of the H\"{o}lder's inequality~\cite[p. 242]{1995bil}. Let's now bound each term appearing on the right-hand side of \eqref{optgapbdexp}. Notice that \begin{align} \label{expsumsqroot} \sum_{j = \ell+1}^{\infty} \left(\mathbb{E}\left[(\delta_{j} + \epsilon_{j})^2 \right]\right)^{1/2} & \leq \sum_{j = \ell+1}^{\infty} \left(2\mathbb{E}[\delta_j^2] + 2\mathbb{E}[\epsilon_j^2] \right)^{1/2} \nonumber \\ & \leq \sum_{j = \ell+1}^{\infty} \frac{1}{\sqrt{m_j}}\left(2\kappa_0(1+c_1+2\sqrt{c_1}) + 2\nu^2\sigmamax^2 \right)^{1/2} \leq \tilde{\kappa}_1\,c_1^{-\ell/2}\end{align} where $\tilde{\kappa}_1 := \left( \frac{1}{\sqrt{m_1}}\frac{\sqrt{c_1}}{\sqrt{c_1}-1}\right)\left(2\kappa_0(1+c_1+2\sqrt{c_1}) + 2\nu^2\sigmamax^2 \right)^{1/2},$ the second inequality in \eqref{expsumsqroot} follows from Assumption~\ref{ass:varbd} and the definition $\epsilon_{\ell}:= \nu\, m_{\ell}^{-1/2}\mbox{proj}(\hat{\sigma}_{\ell},[\sigmamin,\sigmamax])$, and the last inequality follows from using the assumed sample size increase \eqref{sampsizegeoinc}. Similarly, we also get \begin{align} \label{expsumsq} \sum_{\ell = L}^{\infty} \left(\mathbb{E}\left[(\delta_{j} + \epsilon_{j})^2 \right]\right) & \leq \tilde{\kappa}_2\,c_1^{-L},\end{align} where $\tilde{\kappa}_2 := \left( \frac{1}{m_1}\frac{c_1^2}{c_1-1}\right)\left(2\kappa_0(1+c_1+2\sqrt{c_1}) + 2\nu^2\sigmamax^2\right).$ Use \eqref{expsumsqroot} and \eqref{expsumsq} in \eqref{optgapbdexp} to get:
\begin{align} \label{expsqsum} \mathbb{E}\left[\left(c^T\hat{x}^L + q(\hat{x}^{L}) - z^*\right)^2\right] &\leq 4\sum_{\ell=L}^{\infty} \mathbb{E}\left[(\delta_{\ell} + \epsilon_{\ell})^2 \right] + 8 \sum_{\ell = L}^{\infty} \left(\mathbb{E}\left[(\delta_{\ell} + \epsilon_{\ell})^2 \right]\right)^{1/2} \sum_{j = \ell+1}^{\infty} \left(\mathbb{E}\left[(\delta_{j} + \epsilon_{j})^2 \right]\right)^{1/2} \nonumber \\ &\leq 4\tilde{\kappa}_2c_1^{-L} + 8\sum_{\ell = L}^{\infty} \left(\mathbb{E}\left[(\delta_{\ell} + \epsilon_{\ell})^2 \right]\right)^{1/2}\,\tilde{\kappa}_1c_1^{-\ell/2} \nonumber \\ & \leq 4\tilde{\kappa}_2c_1^{-L} + 8(\sqrt{c_1}-1)\,\tilde{\kappa}_1^2\sum_{\ell = L}^{\infty} c_1^{-\ell}= c_1^{-L}\left(4\tilde{\kappa}_2 + \frac{8c_1\tilde{\kappa}_1^2}{\sqrt{c_1}+1} \right).
\end{align} 
Finally, we put it all together to get 
\begin{align}\label{workineq} 
\MoveEqLeft \mathbb{E}\left[\sqrt{W_L}\right]\mathbb{E}\left[\left(c^T\hat{x}^L + q(\hat{x}^{L}) - z^*\right)\right] & \nonumber \\ & \leq \frac{\lambda}{\nu\sigmamin}\left(18\kappa_0(1+c_1+2\sqrt{c_1}) + 2\nu^2c_1\sigmamax^2 \right)^{\frac{1}{2}} \frac{\sqrt{m_1}}{\sqrt{c_1}-1}\left(1 - \frac{1}{c_1^{L/2}}\right)\left(4\tilde{\kappa}_2 + \frac{8c_1\tilde{\kappa}_1^2}{\sqrt{c_1}+1} \right)^{1/2} \nonumber \\ &\leq \frac{\lambda}{\nu\sigmamin}\left(18\kappa_0(1+c_1+2\sqrt{c_1}) + 2\nu^2c_1\sigmamax^2 \right)^{\frac{1}{2}} \frac{\sqrt{m_1}}{\sqrt{c_1}-1}\left(4\tilde{\kappa}_2 + \frac{8c_1\tilde{\kappa}_1^2}{\sqrt{c_1}+1} \right)^{1/2} =: \tau_0, \nonumber \\
\end{align} 
where the first and second inequalities above follow from applying the bounds in \eqref{expsqsum} and \eqref{worksep} and simplifying. This proves the first assertion of the theorem. The second assertion follows simply from the first assertion and the assumed minimum growth rate of the objective function as expressed through Assumption~\ref{ass:fungrowth}.

\end{proof}

The following observations on Theorem~\ref{geometric} are noteworthy.

\begin{enumerate} 
\item[(a)] The assertions in Theorem~\ref{geometric} should be seen as the analogue of the $\mathcal{O}(1/\epsilon^2)$ complexity result in non-smooth convex optimization that is known to be optimal~\cite{2018nes} to within a constant factor. 
\item[(b)] The complexity result in Theorem~\ref{geometric} has been stated in the general population context. So, the result equally applies for the finite-population scenario $|\Xi| < \infty$, although there is strong evidence that in the finite and the countably infinite populations, the best achievable complexity rates may be much faster due to the existence of sharp minima of the sort discussed in~\cite{2000shahom}.  
\item[(c)] The theorem assumes that the sample size schedule $(m_{\ell})_{\ell \geq 1}$ increases geometrically with common ratio $c_1$. Importantly, the result can be generalized in a straightforward manner to a sample size schedule having a stochastic common ratio $C_1$ that is allowed to vary between two deterministic bounds $c_{0}$ and $c_h$ such that $1 < c_0 \leq c_h < \infty$ (see Section~\ref{sec:numerical}). \end{enumerate}

Recall again that the complexity result in Theorem~\ref{geometric} has been obtained assuming that the sample sizes increase geometrically, that is, $m_{\ell}/m_{\ell - 1} = c_1 \in (1,\infty)$, ignoring non-integrality. Can a similar complexity be achieved using other sample size schedules? The following negative result explains why using a slower sample size schedule is bound to result in an inferior complexity. 

\begin{theorem}\label{polynomial} Suppose Assumption~\ref{ass:boundedfeasible}--\ref{ass:fungrowth} hold. Also, suppose there exists $\tilde{\eta}$ such that 
\begin{equation}\label{deltasqlb}
\E\left[\left(\emph{\mbox{dist}}(\s{S}_{m_\ell}, \mathcal{S}^*)\right)\right] \geq \frac{\tilde{\eta}}{\sqrt{m_{\ell}}}.
\end{equation} 
If the sample size schedule is polynomial, that is, \begin{equation}\label{sampsizepolyinc} m_\ell = c_0\, \ell^p, \quad c_0 \in (0, \infty), p \in [1,\infty) \tag{SS-D}. \end{equation}Then there exists $\tau_1 > 0$ such that for $L \geq 3$,
\begin{equation}\label{workcompneg}
\E\left[\emph{\mbox{dist}}(\hat{x}^{L}, \mathcal{S}^*)\right] \geq \frac{\tau_1}{\E\left[W_L^{\frac{1}{2}-\frac{1}{2(1+p)}}\right]}.
\end{equation}
\end{theorem} 
\begin{proof}  
The structure of the algorithm is such that each outer iteration consists of at least one inner iteration. Hence $\tilde{W}_{\ell} \geq m_{\ell},$ implying that 
\begin{equation}\label{totwork} 
W_L \geq \sum_{\ell =1}^L c_0 \ell^p  \geq \int_{1}^L c_0 (\ell-1)^p d\ell = \frac{c_0}{p+1}\left((L-1)^{p+1} - 1\right).
\end{equation} 
Since \eqref{sampsizepolyinc} has been assumed, $m_L = c_0L^p$ and \eqref{totwork} implies, after some algebra, that for $L\geq 3$, \begin{align}\label{totworkmore} W_L &\geq \frac{c_0}{p+1}\left(\frac{m_L}{c_0}\right)^{1+1/p}\left(\left(1-\left(\frac{c_0}{m_L}\right)^{1/p}\right)^{p+1} - \left(\frac{c_0}{m_L}\right)^{1+1/p} \right) \nonumber \\ 
&\geq \frac{c_0}{p+1}\left(\frac{m_L}{c_0}\right)^{1+1/p}\left(\left(1-L^{-1}\right)^{p+1} - L^{-(p+1)} \right) \geq \tau_p\frac{c_0}{p+1}\left(\frac{m_L}{c_0}\right)^{1+1/p}, \end{align} where $\tau_p := \left(\frac{2}{3}\right)^{p+1} - \left(\frac{1}{3}\right)^{p+1}$. Continuing from \eqref{totworkmore}, we get \begin{equation}\label{totworkfinal} W_L^{\frac{1}{2}-\frac{1}{2(1+p)}} \geq \left(\tau_p\frac{c_0}{p+1}\right)^{\frac{p}{2(p+1)}}\sqrt{\frac{m_L}{c_0}}.\end{equation}
Use \eqref{deltasqlb} and \eqref{totworkfinal} to get, for $L \geq 3$, that $\E\left[W_L^{\frac{1}{2}-\frac{1}{2(1+p)}}\right] \,\mathbb{E}\left[\mbox{dist}(\hat{x}^{L}, \mathcal{S}^*) \right] 
\geq \left(\tau_p\frac{c_0}{p+1}\right)^{\frac{p}{2(p+1)}}\frac{\tilde{\eta}}{\sqrt{c_0}}$, thus proving the assertion in the theorem.
\end{proof}

We observe from Theorem~\ref{polynomial} that no matter how large $p \in [1,\infty)$ is chosen when choosing a polynomial sample size schedule, the resulting complexity \eqref{workcompneg} is inferior to the complexity \eqref{workdistcomplexity} implied by a geometric sample size schedule, with the inferiority characterized by the deviation $(2(p+1))^{-1}$. A similar result has been proved by~\cite{2013roysze} in a different context. 

While the results of Theorem~\ref{polynomial} show the superiority of a geometric sequence for the sample size schedule, we emphasize two caveats. First, the lower bound on the (implicit) quality of the sample-path solution set may be violated in, e.g., ``non-quantitative," contexts where the underlying probability space generating the random variables naturally consists of only a finite number of outcomes. The question of what is the best sample size schedule in such contexts is open. Second, we make the obvious observation that during implementation,  considerations other than those included in our analysis, e.g., storage and wall-clock computation time limits, might influence the sample size choice. The conclusions of Theorem~\ref{geometric} and Theorem~\ref{polynomial} should thus be judged within the purview of the analysis considered here.

The condition in \eqref{deltasqlb} might appear cryptic but we believe that this condition will hold under mild conditions. General sufficient conditions under which the sequence $\sqrt{m_{\ell}} \, \mbox{dist}(S_{m_\ell}, \mathcal{S}^*)$ will ``stabilize" to a non-degenerate distribution are well-known~\cite{1993sha,1988dupwet}. Such conditions,  along with assuming the random variables $\sqrt{m_{\ell}} \, \mbox{dist}(S_{m_\ell}, \mathcal{S}^*)$ exhibit uniform integrability, will ensure that the condition in \eqref{deltasqlb} is guaranteed to hold asymptotically.

\section{STOPPING IN FINITE TIME}\label{sec:stopping}
The results we have presented thus far have implied a non-terminating algorithm, as can be seen in the listing of Algorithm~\ref{algo:adaptive-sequential-nonterminating}. Our intent in this section is to demonstrate that the iterates generated by Algorithm~\ref{algo:adaptive-sequential-nonterminating} can be stopped in finite-time while providing a solution with a probabilistic guarantee on the optimality gap. For this, we rely heavily on the finite-stopping results in~\cite{bayraksan2012fixed}. We first describe a simple stopping procedure which is almost identical to what is called FSP in~\cite{bayraksan2012fixed}, and then argue that the stipulations laid out in~\cite{bayraksan2012fixed} hold here, thereby allowing to invoke the main results of~\cite{bayraksan2012fixed}. We note that alternative finite stopping rules have also been studied in the literature, see, e.g., \cite{sen2016mitigating} for a sequential sampling based approach based on the variance associated with 2SLP solutions rather than their corresponding objective values.

Suppose we wish to stop our procedure with a solution whose optimality gap is within $\epsilon > 0$ with probability exceeding $1-\alpha$, $\alpha > 0$. Recall that upon terminating the $\ell$-th outer iteration of Algorithm~\ref{algo:adaptive-sequential}, we have at our disposal an $\mathcal{F}_{\ell}$-measurable candidate solution $\hat{x}^{\ell}$. To construct a one-sided $100(1-\alpha)$ percent confidence interval on the true gap $c^\top \hat{x}^{\ell} + q(\hat{x}^{\ell}) - z^*$, we independently generate an iid sample ${\mathcal{N}}_{\ell} = \{\tilde{\xi}^{\ell}_1, \tilde{\xi}^{\ell}_2, \ldots, \tilde{\xi}^{\ell}_{n_{\ell}}\}$. Assume that the sequence $\{n_{\ell}\}$ of ``testing" sample sizes is non-decreasing; the random objects $\tilde{\xi}^{\ell}_i, i\geq 1, \ell \geq 1$ can be re-used across iterations, that is, $\tilde{\xi}^{\ell}_i$ can be chosen so that if $i < j$ then $\tilde{\xi}^i_k = \tilde{\xi}^j_k$ for $k=1,2,\ldots, n_i$. We then use the set $\mathcal{N}_{\ell}$ to calculate a gap estimate $\tilde{G}^{\ell}_{n_\ell}(\hat{x}^{\ell})$ and sample variance $\tilde{s}^2_{n_\ell}(\hat{x}^{\ell})$ as follows: 
\begin{align}\label{gap-est}
\tilde{G}^{\ell}_{n_\ell}(\hat{x}^{\ell}) &= c^\top (\hat{x}^{\ell} - \tilde{x}^*_{\ell}) + \frac{1}{n_{\ell}}\sum_{i=1}^{n_{\ell}} [Q(\hat{x}^{\ell},\tilde{\xi}^{\ell}_i) - Q(\tilde{x}^*_{\ell},\tilde{\xi}^{\ell}_i)]; \nonumber \\
\tilde{s}^2_{n_\ell}(\hat{x}^{\ell}) &= \frac{1}{n_{\ell}}\sum_{i=1}^{n_{\ell}} \left[Q(\hat{x}^{\ell},\tilde{\xi}^{\ell}_i) - Q(\tilde{x}^*_{\ell},\tilde{\xi}^{\ell}_i) -  \frac{1}{n_{\ell}}\sum_{i=1}^{n_{\ell}} [Q(\hat{x}^{\ell},\tilde{\xi}^{\ell}_i) - Q(\tilde{x}^*_{\ell}, \tilde{\xi}^{\ell}_i)]\right]^2,
\end{align} where $\tilde{x}^*_{\ell}$ is an optimal solution to the sample-path problem $(P_{\ell})$ generated with sample $\mathcal{N}_{\ell}$, and $\delta>0$ is the thresholding constant from Algorithm~\ref{algo:adaptive-sequential-nonterminating}. 

\begin{algorithm}
\caption{An adaptive sequential SAA framework with a finite stopping criterion.}
\label{algo:adaptive-sequential}
\begin{algorithmic}[1]
\STATE{{\bf Input:} Solver-$\mathcal{A}$, a sampling policy, a constant $\nu > 0$, and a constant $\sigmamax > 0$. Set $\ell\leftarrow 0$.}
 \WHILE{$\tilde{G}^{\ell}_{n_\ell}(\hat{x}^{\ell}) + z_{\alpha} \frac{\scriptsize{\mbox{max}}(\tilde{s}_{n_\ell}(\hat{x}^{\ell}),\sigmamax)}{\sqrt{n_{\ell}}} > \epsilon$}
       \STATE{Select the sample size $m_{\ell}$ and draw a random sample $\MM_{\ell} := \{\xi^{\ell}_1,\xi^{\ell}_2,\ldots,\xi^{\ell}_{m_{\ell}}\}$.}
 	\FOR{$\innerit = 1,2,\cdots$} 	 	
 		\STATE{Use Solver-$\mathcal{A}$, e.g., the adaptive partition-based level decomposition~\cite{vanAckooij_Oliveira_Song_2016}, to execute the $\innerit$-th inner iteration for solving the sample-path problem.}
 		\STATE{If $G^{\ell,\innerit} \leq \epsilon_{\ell,t} := \nu \max\left\{\hat{\mbox{se}}_{\ell,\innerit}, \frac{\sigmamax}{\sqrt{m_\ell}}\right\}$, break the inner loop with a candidate solution $\hat{x}^\ell$.}
 	\ENDFOR
 	\STATE{Generate a Monte Carlo sample $\MN_\ell := \{\tilde{\xi}^\ell_1,\tilde{\xi}^\ell_2,\ldots, \tilde{\xi}^\ell_{n_\ell}\}$ (independent from $\MM_\ell$) of sample size $n_\ell$, solve the corresponding sample-path problem $(P_{\ell})$, and calculate $\tilde{G}^{\ell}_{n_\ell}(\hat{x}^{\ell})$ and $\tilde{s}^2_{n_\ell}(\hat{x}^{\ell})$ according to \eqref{gap-est}, respectively.}
 \ENDWHILE
\end{algorithmic}
\end{algorithm} The proposed one-sided $100(1-\alpha)$ percent confidence interval on $\mu(\hat{x}^{\ell}) = c^\top \hat{x}^{\ell} + q(\hat{x}^{\ell}) - z^*$ is
\[
\left[0,\tilde{G}^{\ell}_{n_\ell}(\hat{x}^{\ell}) + z_{\alpha} \frac{\mbox{max}(\tilde{s}_{n_\ell}(\hat{x}^{\ell}),\sigmamax)}{\sqrt{n_{\ell}}}\right],
\] where $z_{\alpha} = \Phi^{-1}(1-\alpha)$ is the $1-\alpha$ quantile of the standard normal distribution, implying that the finite-time procedure stops at iteration $$L(\epsilon) := \arginf_{\ell \geq 1}\left\{\ell: \tilde{G}^{\ell}_{n_\ell}(\hat{x}^{\ell}) + z_{\alpha} \frac{\mbox{max}(\tilde{s}_{n_\ell}(\hat{x}^{\ell}),\sigmamax)}{\sqrt{n_{\ell}}} \leq \epsilon\right\}.$$  Algorithm~\ref{algo:adaptive-sequential} lists a terminating version of Algorithm~\ref{algo:adaptive-sequential-nonterminating} (modulo setting $\sigmamin =0$) based on the proposed confidence interval. The factor $\sigmamax \, n_{\ell}^{-1/2}$ is a thresholding term that is common in sequential settings~\cite{1965chorob} and plays the same role as the term $h(n_k)$ in~\cite{bayraksan2012fixed}, ensuring that $L(\epsilon) \to \infty$ as $\epsilon \to 0$. To analyze the behavior of the coverage probability obtained from Algorithm~\ref{algo:adaptive-sequential}, the following three assumptions are made in~\cite{bayraksan2012fixed}.
\begin{enumerate} 
\item[(A1)] Event $A_{n_{\ell}} = \{\mathcal{S}_{n_{\ell}} \subseteq \mathcal{S}^*\}$ happens with probability $1$ as $\ell \to \infty$. 
\item[(A3)] $\lim_{\ell \to \infty} \mathbb{P}\left\{\sup_{x \in \mathcal{X}} |\tilde{G}^{\ell}_{n_\ell}(\hat{x}^{\ell}) - \mu(x)| > \beta \right\} = 0$ for any $\beta > 0$.
\item[(A4)] $\lim_{\ell \to \infty} \mathbb{P}\left\{\sup_{x \in \mathcal{X}} n_{\ell}^{-1/2}\mbox{max}(\tilde{s}_{n_\ell}(\hat{x}^{\ell}),\sigmamax) > \beta \right\} = 0$ for any
$\beta > 0$. \end{enumerate} (We have omitted (A2) above to preserve the numbering in~\cite{bayraksan2012fixed}.) Theorem 2.3 in~\cite{2000shahom} implies that Assumption (A1) is satisfied if the support $\Xi$ is finite, in addition to Assumptions~\ref{ass:boundedfeasible}--\ref{ass:fungrowth}. Also, it is seen that Assumption (A3) and (A4) hold if the standing Assumption~\ref{ass:varbd} holds. The following result characterizes the behavior of the iterates obtained from Algorithm~\ref{algo:adaptive-sequential}, along with a probabilistic guarantee. We provide a proof only for the third part of the theorem since proofs for the rest either follow trivially or are almost identical to that in~\cite{bayraksan2012fixed}.

\begin{theorem}\label{thm:stopping} Suppose Assumptions~\ref{ass:boundedfeasible}--\ref{ass:fungrowth} hold. Furthermore, let $|\Xi| < \infty$.  Let $m_{\ell}$ and $n_{\ell}$ be positive nondecreasing sequences such that $m_{\ell} \to \infty$ and $n_{\ell} \to \infty$ as $\ell \to \infty$. Then the following assertions hold. 
\begin{enumerate} 
\item $L(\epsilon) < \infty$ a.s. for all $\epsilon>0$  and $L(\epsilon) \to \infty$ a.s. as $\epsilon \to 0$.
 \item Recalling the optimality gap $\mu(x) := c^\top x + q(x) - z^*$, 
 \begin{equation}\label{probguarantee}
 \lim_{\epsilon \to 0} \mathbb{P}\left\{\mu(\hat{x}^{L(\epsilon)}) \leq \epsilon \right\} = 1.
 \end{equation} 
 \item Suppose $\{n_{\ell}\}$ is chosen so that $\liminf_{\ell \to \infty} n_{\ell-1}/n_{\ell} >0$. Then we have that $$\lim_{\epsilon \to 0^+} \epsilon^{2} n_{L(\epsilon)} = O(1).$$ 
 \end{enumerate} 
 \end{theorem} 
 \begin{proof} (Proof of 3.) Following the proof of Lemma 5 in~\cite{bayraksan2012fixed}, we see that there exists $\epsilon_0 > 0$ such that for all $0 < \epsilon < \epsilon_0$, 
\begin{equation}\label{gapvanish}\tilde{G}^{L(\epsilon)}_{n_{L(\epsilon)}}(\hat{x}^{L(\epsilon)}) = 0; \quad \tilde{s}^2_{n_{L(\epsilon)}}(\hat{x}^{L(\epsilon)}) = 0,\end{equation}
where $\tilde{G}^{L(\epsilon)}_{n_{L(\epsilon)}}(\hat{x}^{L(\epsilon)})$ and $\tilde{s}^2_{n_{L(\epsilon)}}(\hat{x}^{L(\epsilon)})$ are from \eqref{gap-est} at stopping.
 According to the stopping criterion of Algorithm~\ref{algo:adaptive-sequential}, we have that:
\begin{align}\label{sampbds}
\epsilon^2 n_{L(\epsilon)} & \geq \left(\sqrt{n_{L(\epsilon)}}\tilde{G}^{L(\epsilon)}_{n_{L(\epsilon)}}(\hat{x}^{L(\epsilon)}) + z_{\alpha} \mbox{max}(\tilde{s}_{n_{L(\epsilon)}}(\hat{x}^{L(\epsilon)}),\delta) \right)^2; \nonumber \\ \epsilon^2 n_{L(\epsilon)-1} & \leq \left(\sqrt{n_{L(\epsilon)-1}}\tilde{G}^{L(\epsilon)-1}_{n_{L(\epsilon)-1}}(\hat{x}^{L(\epsilon)-1}) + z_{\alpha} \mbox{max}(\tilde{s}_{n_{L(\epsilon)-1}}(\hat{x}^{L(\epsilon)-1}),\delta) \right)^2.
\end{align} Now notice that since $\liminf_{\ell \to \infty} n_{\ell-1} / n_{\ell} > 0$ and $L(\epsilon) \to \infty$ as $\epsilon \to 0$ a.s., there exists $\tilde{\beta} > 0$ such that for small enough $\epsilon$, we have \begin{equation}\label{succsampsizerel} n_{L(\epsilon) - 1} \geq \tilde{\beta} \, n_{L(\epsilon)} \mbox{ a.s.}\end{equation} 
Using \eqref{succsampsizerel}, \eqref{sampbds}, and \eqref{gapvanish}, we get, a.s., $z_{\alpha} \delta^2 \leq \lim_{\epsilon \to 0^+} \frac{n_{L(\epsilon)}}{1/\epsilon^2} \leq \frac{z_{\alpha}}{\tilde{\beta}} \, \delta^2$.
\end{proof}

It is worth noting that the main probabilistic guarantee appearing in \eqref{probguarantee} is stronger than classical guarantees in sequential testing such as those in~\cite{1965chorob}. This deviation from a classical stopping result is primarily because of the fast convergence assured by (A1). It is possible and likely that when (A1) is relaxed, a more classical result such as what one encounters in~\cite{1965chorob} holds, but we are not aware of the existence of such a result.  

The condition $\liminf_{\ell \to \infty} n_{\ell-1} / n_{\ell} > 0$ stipulated by the third assertion of Theorem~\ref{thm:stopping} is satisfied by a wide variety of sequences. For instance, if $q_0, q_1 \in (0,\infty)$, any logarithmic increase schedule $n_{\ell} = q_0 + q_1\log \ell$, any polynomial increase schedule $n_{\ell} = q_0 + q_1 \ell^p, p \in (0,\infty)$, and any geometric increase schedule $n_{\ell}/n_{\ell-1} = q_1$ satisfy the condition $\liminf_{\ell \to \infty} n_{\ell-1} / n_{\ell} > 0$.

\section{COMPUTATIONAL EXPERIMENTS}\label{sec:numerical}
In this section, we present computational results of the proposed adaptive sequential sampling framework for solving 2SLPs with fixed recourse and fixed second-stage objective coefficients. We chose problems instances of this type to enable a ``warm starting" procedure, where the initial solution and an initial second-stage value function approximation for every sample-path problem at each outer iteration can be obtained using information gained from previous iterations. (This procedure is summarized in Algorithm C.1 in the appendix of the online supplementary document~\cite{online}.) For the purpose of benchmarking, we consider finite-sample instances of such problems, that is, problems where $|\Xi| < \infty$, so that we get access to the true optimal value $z^*$ up to a pre-specified precision by solving these instances using a deterministic solver. In particular, we apply the adaptive partition-based level decomposition method~\cite{vanAckooij_Oliveira_Song_2016}, which has shown to be a competitive state-of-the-art solution approach. Five finite-sample instances of each problem in a selected problem class are generated; $20$ replications of each competing sequential SAA algorithm are performed on each of the generated problem instances (except for the ssn instances, where only $10$ replications are performed due to the extensive computational effort for solving these instances). We implemented all algorithms in C++ using the commercial solver CPLEX, version 12.8. All tests are conducted on an iMac desktop with four 4.00GHz processors and 16Gb memory. The number of threads is set to be one.

We run the adaptive sequential SAA framework according to Algorithm~\ref{algo:adaptive-sequential}, and record the total number of outer iterations as $L$, the final candidate solution at the $L$-th iteration as $\hat{x}^L$, and the sample size used in the final iteration $L$ as $N_L$; $c^\top \hat{x}^L + q(\hat{x}^L)$ then gives the true objective value of final candidate solution $\hat{x}^L$. We report in column ``CI'' the ratio between the width of the reported confidence interval (at stopping) for the optimality gap and the true objective value corresponding to $\hat{x}^L$. The threshold $\epsilon$ is chosen to be small enough relative to the objective value corresponding to the candidate solution obtained from the outer iteration, e.g., $10^{-3}\times \left(c^\top \hat{x}^1 + Q^1_{m_1}(\hat{x}^1)\right)$. After Algorithm~\ref{algo:adaptive-sequential} terminates with a final solution $\hat{x}^L$, we verify whether or not the true optimal objective value $z^*$ is in the reported confidence interval. Since the confidence interval at stopping is guaranteed to cover $z^*$ only asymptotically (see Theorem~\ref{thm:stopping}), we report the coverage probability at stopping in the column titled ``cov.'', using results obtained from the $20$ replications for each test instance except  ssn and 20term, where $10$ replications are used. 

We set the sample size $m_\ell$ for the $\ell$-th sample-path problem to be twice as large as the sample size $n_\ell$ for validating the quality of candidate solution $\hat{x}^\ell$, i.e., $m_\ell = 2\times n_\ell, \ \forall \ell = 1,2,\ldots$. This choice is motivated by the practical guideline~\cite{bayraksan2011sequential} that the computational effort expended to find candidate solutions should be higher than that expended to compare candidate solutions. The following additional notation is used in the tables that follow.
\begin{itemize}
\item Time: computational time (recorded in seconds)
\item $M$: total number of inner iterations.
\item $L$: total number of outer iterations.
\item $n_L$: the sample size used in the final outer iteration $L$.
\end{itemize}

\subsection{Implementation details}
The following five algorithms are implemented in our computational study. The procedures described in (iii), (iv), and (v) use Algorithm~\ref{algo:adaptive-sequential} with different sample size schedules. The procedure listed in (i) has been shown to be very competitive recently; the procedure in (ii) is proposed in~\cite{bayraksan2012fixed}.
\begin{itemize}
\item[(i)]  {\bf PILD-ODA}. This algorithm is the adaptive partition-based level decomposition algorithm with on-demand accuracy as proposed in~\cite{vanAckooij_Oliveira_Song_2016}, which is used to solve each instance with the full set of scenarios up to a relative optimality gap of $10^{-4}$. Note that $z^*$ for each instance is also obtained by this algorithm using a smaller relative optimality gap threshold of $10^{-6}$.

\item[(ii)] {\bf Sequential-BP-L($\Delta$)}. This algorithm follows the sampling schedules in~\cite{bayraksan2012fixed} while solving individual sample-path problems to high precision. Specifically, each sample-path problem (with a sample size of $m_\ell$) is solved up to a relative optimality gap of $10^{-6}$ in each outer iteration $\ell$, using a standard level decomposition approach for solving 2SLPs~\cite{fabian2007solving}. Note that our implementation of this approach does not incorporate the warm starting functionality. The obtained candidate solution $\hat{x}^\ell$ is then evaluated using a sample of size $n_\ell$. To obtain $x^*_{n_\ell}$ that appears in $\tilde{G}^\ell_{n_\ell}$ and $\tilde{s}^2_{n_\ell}$ in \eqref{gap-est}, we solve the corresponding sample-path problem up to a relative optimality gap of $10^{-4}$, as suggested by \cite{bayraksan2012fixed}. By default, we use a linear sample size schedule where $\Delta = 100$ additional scenarios are sampled from one iteration to the next, starting with an initial sample size $m_1 = 2\times n_1 = 100$. We use the same initial sample size for all variants of the sequential sampling approaches that we describe below, although one may tune this parameter for further enhancements.

\item[(iii)] {\bf Adaptive-seq-BP-L($\Delta$)}. This is Algorithm~\ref{algo:adaptive-sequential} implemented with the linearly increasing sample size schedule proposed in~\cite{bayraksan2012fixed}, that is, $m_{\ell+1} = m_{\ell} + \Delta$. For ``warm starting" the initial solution and an initial second-stage value function approximation for every sample-path problem at each outer iteration, we use Algorithm C.1 in the appendix of the online supplementary document~\cite{online}. We use parameter $\alpha = 0.1$ and safeguard parameter $\delta = 10^{-5}$ in defining the adaptive optimality tolerance $\epsilon_\ell$ according to \eqref{innerterm}. PILD-ODA is applied to solve each sample-path problem with the aforementioned warm starting functionality. 

\item[(iv)] {\bf Adaptive-seq-fixed ($c_1$)}. This is Algorithm~\ref{algo:adaptive-sequential} implemented with a geometric sample size schedule. The setting is nearly identical to (iii) except that we use a fixed rate $c_1$ as the geometric increase rate, that is, $m_{\ell+1} = c_1\, m_{\ell}$. 

\item[(v)] {\bf Adaptive-seq-dyn($c_0,c_h$)}. Like in (iv), this is Algorithm~\ref{algo:adaptive-sequential} implemented with a geometric sample size schedule ensuring that $m_{\ell+1} = C_1\, m_{\ell}$. However, unlike in (iv), the rate $C_1$ is dynamic (and hence, listed in uppercase) within chosen bounds $c_0, c_h$.  Specifically, starting from some initial value of $C_1$, if the inner loop finishes after a single iteration, implying that the problem with the current sample size does not deviate much from the one solved in the previous outer iteration, we increase the deviation of $C_1$ from $1$ by a factor of $2$ subject to $C_1$ not exceeding $c_h$. Formally, we set $C_1 \leftarrow \mbox{min}(2C_1 - 1,c_h)$. If, on the other hand, the inner loop takes more than four iterations, we shrink the deviation of $C_1$ from $1$ by a factor of $2$, subject to $C_1$ reaching a minimum of $c_0$, that is, we set $C_1 \leftarrow \mbox{max}(c_0,\frac{1}{2}C_1 + \frac{1}{2})$. While our theory does not explicitly cover this ``dynamic $C_1$" context, an extension of our theory to this case is straightforward. See comment (c) appearing after Theorem~\ref{geometric}.   

\end{itemize}

In all algorithms that we tested except ``PILD-ODA,'' we use a time limit of two hours ($7200$ seconds). When the stopping criterion is not met by the time limit, we report the smallest value $\tilde{G}^\ell_{n_\ell}(\hat{x}^\ell) + z_\alpha \frac{\mbox{\scriptsize{max}}(\tilde{s}_{n_\ell}(\hat{x}^\ell),\sigma_{\max})}{\sqrt{n_\ell}}$ encountered during all completed outer iterations $\ell$, and accordingly consider this quantity the width of the confidence interval on the optimality gap of $\hat{x}^{\ell}$. The profiles of test instances used in our computational experiments are summarized in Table~\ref{testinstances}, where the set of DEAK instances are randomly generated test instances from~\cite{deak2011testing}, and other instances are taken from existing literature that are linked to certain ``real-world'' applications. For the purpose of benchmarking, we also create an additional family of instances based on the DEAK instances by increasing the variance of the underlying random variables generating the test instances. We use ``High'' to label this new set of DEAK instances with higher variance in Table~\ref{initialtest},~\ref{testproposed1}, and~\ref{testproposed2}.
\begin{table}\label{testinstances}
\caption{Profiles of test instances from the literature. Notation $(n_a,n_b)$ means that the number of variables is given by $n_a$ and the number of constraints is given by $n_b$.}\label{instance-size}
{\footnotesize 
\begin{center}
\begin{tabular}{c|c|c|c}
    \hline\noalign{\smallskip}
    Instance & First-stage size & Second-stage size & Reference\\
    \noalign{\smallskip}\hline\noalign{\smallskip}
    DEAK40$\times$20 & (40,20) & (30,20) & \cite{deak2011testing} \\
    DEAK40$\times$40 & (40,20) & (60,40) & -\\
    DEAK40$\times$60 & (40,20) & (90,60) & -\\
    DEAK60$\times$20 & (60,30) & (30,20) & -\\
    DEAK60$\times$40 & (60,30) & (60,40) & - \\
    DEAK60$\times$60 & (60,30) & (90,60) & -\\
    LandS & (4,2) & (12,7) & \cite{LandS} \\
    gbd & (17,4) & (10,5) & \cite{ferguson1956allocation} \\
    4node & (52,14) & (186,74) & \cite{Ariyawansa_Felt_2004} \\
    pgp2 & (4,2) & (16,7) & \cite{higle2013stochastic} \\
    retail & (7,0) & (70,22) & \cite{herer2006multilocation} \\
    cep & (8,5) & (15,7) & \cite{higle2013stochastic}\\
    baa99-20 & (20,0) & (250,40) & \cite{sen2016mitigating}\\ 
    20-term & (63,3) & (764,124) & \cite{mak1999monte} \\
    ssn & (89,1) & (706,175) & \cite{sen1994network} \\
    \noalign{\smallskip}\hline
\end{tabular}
\end{center}
}
\end{table}

\subsection{Numerical results}
We first investigate the empirical performance of ``Sequential-BP-L($\Delta$)'', and its adaptation ``Adaptive-seq-BP-L($\Delta$)'' into our proposed framework, against ``PILD-ODA'' which is arguably a state-of-the-art approach for solving 2SLPs with fixed recourse and fixed second-stage objective coefficients using the full set of scenarios~\cite{vanAckooij_Oliveira_Song_2016}.  Table~\ref{initialtest} summarizes the results on our test instances. We recall that for all the sequential SAA approaches, the numbers shown in each row are calculated by taking the average of the corresponding values over $20$ replications ($10$ replications for ssn and 20term) of algorithm instantiation on five finite-sample instances.

\subsubsection{Computational results on the DEAK instance family} 
We first present the performance of aforementioned algorithms on the DEAK instance family. Instances within this family share the same structure and vary by the problem sizes in terms of the number of variables and constraints. Experiments on these different instances allow us to see how the algorithms behave as the problem sizes change given the same underlying problem structure.

From Table~\ref{initialtest}, we see that sequential SAA algorithms ``Sequential-BP-L($100$)'' and ``Adaptive-seq-BP-L($100$)'' are clearly favored over the direct approach ``PILD-ODA.''  The sequential SAA approaches finish in much less computational time at a low price in terms of optimality gap --- around 0.1\%. The coverage probabilities of these approaches are also satisfactory. The majority of the computational savings come from the fact that sequential SAA approaches expend much less effort in each inner iteration, since only a (small) sample is taken at each early outer iteration $\ell$.

In comparing ``Sequential-BP-L($\Delta$)'' against ``Adaptive-seq-BP-L($\Delta$),'' notice from Table~\ref{initialtest} that the computational time for ``Adaptive-seq-BP-L($\Delta$)'' is lower in most cases, while the total number of outer iterations $L$, inner iterations $M$, and the final sample size $n_L$ are similar. This is again explainable since in ``Sequential-BP-L,'' the sample-path problems in each outer iteration are solved to a high precision, whereas in ``Adaptive-seq-BP-L($\Delta$),'' the sample-path problems are only solved up to a factor of the sampling error as detailed in Algorithm \ref{algo:adaptive-sequential}. Furthermore, a warm start functionality and an adaptive scenario aggregation technique are leveraged in ``Adaptive-seq-BP-L($\Delta$),'' by using Algorithm C.1 in the appendix of the online supplementary document~\cite{online} and PILD-ODA~\cite{vanAckooij_Oliveira_Song_2016}, respectively.


\begin{table}\label{initialtest}
\caption{Computational results of the adaptive partition-based level decomposition approach \cite{vanAckooij_Oliveira_Song_2016} (``PILD-ODA''), the sequential sampling procedure by \cite{bayraksan2012fixed} (``Sequential-BP-L''), and Algorithm \ref{algo:adaptive-sequential} with the stopping criterion and sample size schedule proposed in \cite{bayraksan2012fixed} (``Adaptive-seq-BP-L (100)'') on our test instances DEAK and DEAK-H.}
\begin{footnotesize}
\begin{tabular}{lllllllllllll}
\hline
Ins & $N$ & & \multicolumn{2}{c}{PILD-ODA}  & &     \multicolumn{3}{c}{Sequential-BP-L($100$)}  & & \multicolumn{3}{c}{Adaptive-seq-BP-L($100$)}  \\
 \cline{4-5} \cline{7-9} \cline{11-13}
                           &    &  & Time & M           &     & Time       & $M(L,n_L)$ & CI (cov.) &   & Time       & $M(L,n_L)$ & CI (cov.)    \\
                      \cline{4-13}
\multirow{3}{*}{40x20} & 50K &  & 53.4 & 19 & & 5.4 & 14(5,1070) & (0.1,97)& & 1.5 & 20(5,1094) & (0.1,97)\\
 & 100K &  & 101.8 & 18 & & 5.1 & 13(5,1032) & (0.1,99) & & 1.3 & 19(5,1014) & (0.1,97)\\
\hline
\multirow{3}{*}{40x40} & 50K &  & 74.6 & 12  & & 4.3 & 19(3,584) & (0.0,83) & & 1.2 & 12(3,630) & (0.1,80)\\
 & 100K &  & 134.1 & 12 & & 5.6 & 20(3,660) & (0.1,90) & & 1.3 & 13(3,676) & (0.1,82) \\
\hline
\multirow{3}{*}{40x60} & 50K &  & 206.2 & 19  & & 4.3 & 20(2,374) & (0.1,96)  & & 1.7 & 21(2,396) & (0.1,100)\\
 & 100K &  & 413.1 & 20 & & 4.1 & 20(2,360) & (0.1,99) & & 1.6 & 21(2,366) & (0.1,100) \\
\hline
\multirow{3}{*}{60x20} & 50K &  & 114.4 & 56  & & 86.1 & 41(13,2540) & (0.1,100) & & 18.5 & 64(13,2596) & (0.1,100)\\
 & 100K &  & 252.2 & 60 & & 87.8 & 42(13,2584) & (0.1,100) & & 19.1 & 64(13,2636) & (0.1,100)\\
\hline
\multirow{3}{*}{60x40} & 50K &  & 502.0 & 65 & & 23.2 & 32(4,824) & (0.1,100) & & 12.3 & 70(4,834) & (0.1,100)\\
 & 100K &  & 929.4 & 67  & & 25.1 & 33(4,864) & (0.1,100) & & 13.5 & 70(4,876) & (0.1,100)\\
\hline
\multirow{3}{*}{60x60} & 50K &  & 333.8 & 24 & & 5.9 & 22(2,414) & (0.1,100) & & 2.2 & 25(2,424) & (0.1,100)\\
 & 100K &  & 622.3 & 24  & & 6.5 & 22(2,436) & (0.1,100) & & 2.3 & 25(2,436) & (0.1,100)\\
\hline
\hline
\multirow{1}{*}{40x20} & 50K &  & 63.9 & 17 & & 18.6 & 27(9,1776) & (0.1,96) & & 4.4 & 23(8,1698) & (0.1,98)\\
High & 100K &  & 139.2 & 18 & & 18.2 & 27(9,1772) & (0.1,96) & & 5.3 & 24(9,1854) & (0.1,95) \\
\hline
\multirow{1}{*}{40x40} & 50K &  & 58.9  & 9  & & 4.5 & 17(3,580) & (0.0,83) & & 1.3 & 10(3,640) & (0.0,70)\\
High & 100K &  & 117.0 & 9 & & 4.0 & 17(3,556) & (0.1,88) & & 1.3 & 10(3,646) & (0.1,80)\\
\hline
\multirow{1}{*}{40x60} & 50K &  & 711.5 & 25  & & 60.8 & 42(6,1140) & (0.1,99) & & 22.4 & 29(6,1132) & (0.1,93)\\
High & 100K &  & 1520.0 & 24  & & 55.5 & 41(6,1102) & (0.1,100) & & 20.8 & 29(6,1130) & (0.1,93) \\
\hline
\multirow{1}{*}{60x20} & 50K &  & 162.6 & 46 & & 139.4 & 53(16,3194) & (0.1,99) & & 78.9 & 52(16,3280) & (0.1,100)\\
High & 100K &  & 263.2 & 43 & & 132.9 & 54(16,3160) & (0.1,100) & & 73.5 & 52(16,3230) & (0.1,100) \\
\hline
\multirow{1}{*}{60x40} & 50K &  & 432.8 & 31 & & 112.6 & 55(9,1824) & (0.1,99) & & 127.4 & 42(10,1920) & (0.1,98)\\
High & 100K &  & 958.5 &32  & & 124.0 & 56(9,1834) & (0.1,99) & & 122.4 & 42(10,1940) & (0.1,98) \\
\hline
\multirow{1}{*}{60x60} & 50K &  & 673.5 & 23 & & 96.6 & 48(6,1290) & (0.1,100) & & 38.2 & 31(6,1282) & (0.1,90) \\
High & 100K &  & 1591.9 & 25  & & 107.2 & 49(7,1316) & (0.1,96) & & 42.3 & 31(7,1362) & (0.1,89)\\
\hline
\end{tabular}
\end{footnotesize}
\end{table}

Table~\ref{initialtest} provides clear evidence of the effectiveness of the sequential SAA framework and the use of warm starts. In an attempt to investigate the effect of geometric sampling schemes, which assuredly preserve the Monte Carlo canonical rate by Theorem~\ref{geometric}, we next compare in Table~\ref{testproposed1} the computational results of the adaptive sequential SAA with a geometric sample size schedule having a fixed increase rate $c_1=1.5$ (option ``Adaptive-seq-fixed($1.5$)'') against a dynamically chosen geometric increase rate with $c_0 = 1.05, c_h = 3$ and $C_1$ starting at $1.5$ (option ``Adaptive-seq-dyn($1.05,3$)''), when employed with a finite-time stopping criterion. We see that similar results are obtained by the two alternative options in terms of the computational time. ``Adaptive-seq-dyn($1.05,3$)'' exhibits slightly fewer inner and outer iterations, whereas the sample sizes seem significantly larger. Also, comparing Table~\ref{initialtest} against Table~\ref{testproposed1}, it seems clear that a geometrically increasing sample size schedule results in a large sample size at stopping but generally fewer outer iterations than the linear increasing rate employed in ``Adaptive-seq-BP-L''. In ``Adaptive-seq-dyn,'' the sample size at stopping is even larger, but the number of outer iterations and the number of inner iterations are reduced, leading to less computational time in general. All options share similar behavior from the standpoint of the width of the confidence interval and its coverage.

\begin{table}\label{testproposed1}
\caption{Computational results of the adaptive partition-based level decomposition approach \cite{vanAckooij_Oliveira_Song_2016} - ``PILD-ODA'', Algorithm \ref{algo:adaptive-sequential} with a fixed increasing rate ``Adaptive-seq-fixed($1.5$),'' and Algorithm \ref{algo:adaptive-sequential} with a dynamic increase rate ``Adaptive-seq-dyn($1.05,3$),'' and with $C_1$ starting at $1.5$ on our test instances DEAK and DEAK-H. }\label{table:4}
\begin{footnotesize}
\begin{tabular}{lllllllllllll}
\hline
Ins & $N$ & & \multicolumn{2}{c}{PILD-ODA}  & & \multicolumn{3}{c}{Adaptive-seq-fixed($1.5$)}      &       & \multicolumn{3}{c}{Adaptive-seq-dyn($1.05,3$)}    \\
 \cline{4-5} \cline{7-9} \cline{11-13}
                           &    &  & Time & M           &     & Time       & $M(L,n_L)$ & CI(cov.) &   & Time       & $M(L,n_L)$ & CI(cov.)    \\
                      \cline{4-13}
\multirow{3}{*}{40x20} & 50K &  & 53.4 & 19 & & 1.5 & 21(7,1377) & (0.1,96) & & 1.6 & 19(4,2892) & (0.1,100) \\

 & 100K &  &  101.8 & 18 & & 1.5 & 21(7,1438) & (0.1,99) & & 1.6 & 19(4,2886) & (0.1,100) \\

\hline
\multirow{3}{*}{40x40} & 50K &  & 74.6 & 12  & & 1.2 & 13(4,568) & (0.1,71) & & 1.8 & 13(4,1662) & (0.0,75) \\
 & 100K &  &134.1 & 12 & & 1.2 & 14(4,595) & (0.1,72) & & 1.7 & 13(3,1489) & (0.0,75) \\

\hline
\multirow{3}{*}{40x60} & 50K &  & 206.2 & 19  & & 1.9 & 23(3,318) & (0.1,100) & & 1.9 & 22(3,454) & (0.1,100) \\

 & 100K &  & 413.1 & 20  & & 1.9 & 23(3,308) & (0.1,100) & & 1.9 & 23(3,458) & (0.1,100) \\

\hline
\multirow{3}{*}{60x20} & 50K &  & 114.4 & 56  & & 10.7 & 60(9,3675) & (0.1,100) & & 9.2 & 56(5,6048) & (0.1,100) \\

 & 100K &  & 252.2 & 60 & & 11.0 & 60(9,3673) & (0.1,100) & & 9.5 & 56(5,6264) & (0.1,100) \\

\hline
\multirow{3}{*}{60x40} & 50K &  & 502.0 & 65 & & 14.1 & 73(6,921) & (0.1,100) & & 13.8 & 69(4,1620) & (0.1,100) \\

 & 100K &  & 929.4 & 67   & & 14.7 & 73(6,959) & (0.1,100) & & 13.4 & 68(4,1566) & (0.1,100) \\

\hline
\multirow{3}{*}{60x60} & 50K &  & 333.8 & 24 & & 2.7 & 28(4,374) & (0.1,100) & & 2.8 & 27(3,617) & (0.1,100) \\

 & 100K &  & 622.3 & 24   & & 2.7 & 28(4,374) & (0.1,100) & & 2.7 & 27(3,580) & (0.1,100) \\

\hline
\hline
\multirow{1}{*}{40x20} & 50K &  &  63.9 & 17 & & 4.4 & 23(9,3034) & (0.1,97) & & 4.0 & 19(5,5400) & (0.1,99) \\
High & 100K &  & 139.2 & 18 & & 4.4 & 23(9,3013) & (0.1,95) & & 5.3 & 20(5,7066) & (0.0,98) \\

\hline
\multirow{1}{*}{40x40} & 50K &  & 58.9  & 9 & & 1.3 & 11(4,617) & (0.0,69) & & 1.8 & 11(4,1485) & (0.0,65) \\
High & 100K &  & 117.0 & 9  & & 1.3 & 11(4,601) & (0.0,61) & & 1.7 & 10(3,1366) & (0.0,65) \\

\hline
\multirow{1}{*}{40x60} & 50K &  & 711.5 & 25 & & 24.6 & 31(7,1535) & (0.1,93) & & 28.1 & 27(4,3240) & (0.1,96) \\
High & 100K &  &  1520.0 & 24  & & 22.0 & 31(7,1427) & (0.1,92) & & 27.0 & 27(4,3046) & (0.1,93) \\

\hline
\multirow{1}{*}{60x20} & 50K &  & 162.6 & 46 & & 38.0 & 46(10,5558) & (0.1,100) & & 34.3 & 43(6,9720) & (0.1,100) \\
High & 100K &  & 263.2 & 43 & & 42.4 & 46(10,6086) & (0.1,100) & & 33.1 & 43(6,9720) & (0.1,100) \\

\hline
\multirow{1}{*}{60x40} & 50K &  & 432.8 & 31 & & 70.4 & 40(9,2866) & (0.1,99) & & 78.1 & 33(5,5706) & (0.1,99) \\
High & 100K &  & 958.5 &32 & & 78.6 & 40(9,2894) & (0.1,98) & & 75.3 & 33(5,5688) & (0.1,96) \\

\hline
\multirow{1}{*}{60x60} & 50K &  & 673.5 & 23 & & 42.3 & 32(7,1878) & (0.1,92) & & 42.2 & 27(5,3831) & (0.1,94) \\
High & 100K &  & 1591.9 & 25  & & 38.4 & 32(7,1808) & (0.1,85) & & 50.9 & 27(5,4078) & (0.1,89) \\

\hline
\end{tabular}
\end{footnotesize}
\end{table}

We next investigate the sensitivity of chosen parameters such as the sample size increase rate for the proposed approaches. We observe from Table~\ref{initialtest} and Table~\ref{testproposed1} that, as opposed to what has been suggested in theory (Theorem~\ref{geometric}), Algorithm~\ref{algo:adaptive-sequential} with a linear sample size schedule performs competitively with the one with a geometric sample size schedule in our test instances. This may be because the algorithm ``Sequential-BP-L($\Delta$)'' in Table~\ref{initialtest} with a value $\Delta=100$ mimics the behavior of a geometric sequence. To validate this suspicion, Table~\ref{testproposed2} presents the performance of ``Adaptive-seq-BP-L($\Delta$)'' implemented with a linear sample size schedule having a smaller increase $\Delta = 10$ and ``Adaptive-seq-fixed($c_1$)'' with a smaller geometric increase rate $c_1 = 1.1$. We also display the performance of ``Adaptive-seq-dyn($c_0,c_h$)'' with $c_0 = 1.05, c_h =2$ and with $C_1$ starting at $1.1$, alongside these algorithms.

\begin{table}\label{testproposed2}
\centering
\caption{Computational results of Algorithm \ref{algo:adaptive-sequential} with the fixed-width stopping criterion and linear sample size schedule proposed in \cite{bayraksan2012fixed} with an increase of $10$ scenarios per iteration (``Adaptive-seq-BP-L($10$)''), Algorithm \ref{algo:adaptive-sequential} with a geometrically increasing sample size schedule with rate $c_1 = 1.1$ (``Adaptive-seq-fixed($1.1$)''), and Algorithm \ref{algo:adaptive-sequential} with a geometrically increasing sample size schedule having a dynamic rate (``Adaptive-seq-dyn($1.05,3$)''), with $C_1$ starting at $1.1$, on our test instances DEAK and DEAK-H.}\label{table:7}
\begin{footnotesize}
\begin{tabular}{lllllllllll}
\hline
Ins & $N$ & & \multicolumn{2}{c}{Adaptive-seq-BP-L($10$)}  & & \multicolumn{2}{c}{Adaptive-seq-fixed($1.1$)}      &       & \multicolumn{2}{c}{Adaptive-seq-dyn($1.05,3$)}    \\
 \cline{4-5} \cline{7-8} \cline{10-11}
                           &    &  & Time & $M(L,n_L)$           &     & Time       & $M(L,n_L)$   &   & Time       & $M(L,n_L)$    \\
                      \cline{4-11}
\multirow{3}{*}{40x20} & 50K &  & 3.2 & 37(23,551)  & & 2.8 & 36(22,760)  & & 1.7 & 21(6,2797) \\

 & 100K &  & 3.5 & 39(24,579) & & 2.7 & 35(21,721)  & & 1.7 & 21(6,2711) \\

\hline
\multirow{3}{*}{40x40} & 50K &  & 1.4 & 17(8,249) & & 1.4 & 19(9,250)  & & 1.7 & 14(5,1319)  \\

 & 100K &  & 1.3 & 17(7,239) & & 1.4 & 19(9,252)  & & 1.5 & 14(5,1143)  \\

\hline
\multirow{3}{*}{40x60} & 50K &  & 2.5 & 27(6,204)  & & 2.6 & 29(7,188) & & 2.3 & 25(4,421) \\
 & 100K &  & 2.1 & 26(5,186) & & 2.7 & 30(7,193)  & & 2.1 & 25(4,369) \\

\hline
\multirow{3}{*}{60x20} & 50K &  & 102.5 & 144(93,1945)  & & 30.8 & 87(36,2760) & & 10.2 & 58(7,6383)  \\

 & 100K &  & 103.8 & 143(93,1936)  & & 31.1 & 87(36,2768) & & 10.5 & 58(7,6435)  \\

\hline
\multirow{3}{*}{60x40} & 50K &  & 47.9 & 92(24,560)  & & 38.2 & 90(21,682)  & & 15.5 & 73(6,1578)  \\

 & 100K &  & 51.3 & 92(24,572)  & & 37.1 & 88(21,665)  & & 16.7 & 72(6,1733)  \\

\hline
\multirow{3}{*}{60x60} & 50K &  & 3.9 & 35(7,233)  & & 4.4 & 38(9,235)  & & 3.0 & 30(4,459) \\
 & 100K &  & 3.7 & 34(7,230)  & & 4.1 & 37(9,222)  & & 3.3 & 30(5,539)  \\
\hline
\hline
\multirow{1}{*}{40x20} & 50K &  & 11.6 & 53(39,875)  & & 7.9 & 42(28,1410) & & 4.2 & 21(7,5371) \\

High & 100K &  & 12.9 & 54(40,891)  & & 9.5 & 44(30,1612)  & & 5.1 & 21(7,6229)  \\

\hline
\multirow{1}{*}{40x40} & 50K &  & 1.4 & 14(8,246) & & 1.4 & 15(9,231) & & 1.5 & 11(5,1030)  \\

High & 100K &  & 1.5 & 14(8,251)  & & 1.5 & 15(9,238) & & 1.4 & 11(5,956)  \\

\hline
\multirow{1}{*}{40x60} & 50K &  & 263.4 & 77(30,683)  & & 78.0 & 65(24,940)  & & 32.8 & 30(6,3237)  \\

High & 100K &  & 200.6 & 73(28,646)  & & 68.3 & 63(23,859)  & & 29.3 & 31(6,2904)  \\

\hline
\multirow{1}{*}{60x20} & 50K &  & 337.7 & 128(94,1951)  & & 101.6 & 73(38,3413)  & & 34.5 & 45(8,9523)  \\

High & 100K &  & 341.5 & 130(96,1988) & & 97.6 & 72(37,3271)  & & 26.8 & 45(7,7979)  \\

\hline
\multirow{1}{*}{60x40} & 50K &  &2283.2 & 141(59,1271)  & & 268.1 & 85(31,1758)  & & 97.6 & 37(7,5817)  \\
High & 100K & & 2075.0 & 133(55,1196)  & & 261.0 & 83(30,1710)  & & 78.8 & 36(7,5363)  \\
\hline
\multirow{1}{*}{60x60} & 50K &  & 742.6 & 88(35,793)  & & 134.3 & 69(26,1106) & & 53.0 & 31(7,3987)  \\

High & 100K &  & 621.0 & 82(32,735)  & & 144.6 & 67(25,1052)  & & 51.0 & 31(7,3593) \\

\hline
\end{tabular}
\end{footnotesize}
\end{table}
Comparing between Table~\ref{testproposed2} and Table~\ref{testproposed1}, we see that the performance of ``Adaptive-seq-BP-L($10$),'' where the sample size increases by $10$ in each iteration, is significantly worse than ``Adaptive-seq-BP-L($100$),'' where the sample size increases by $100$ in each iteration. Although the final sample size $n_L$ is lower at stopping when a slower linear sample size schedule is utilized, this comes at the price of a larger number of outer and inner iterations, leading to substantially more computational time. The same effect happens to option ``Adaptive-seq-fixed($c_1$)'' as well, but at a much less significant level, where utilizing a smaller $c_1$ ends up with a larger number of outer iterations and slightly more computational time. On the other hand, the performance of Algorithm \ref{algo:adaptive-sequential} with a dynamic increase rate (option ``Adaptive-seq-dyn($1.05,3$)'') does not appear to be impacted much from the choice of the starting increasing rate $C_1$.

\exclude{We observe, however, that the performance of option ``Adaptive-seq-dyn($c_0,c_h$)'' is robust to initial value of $C_1$. This is expected because the sample size increase rate is adjusted depending on how many inner iterations were expended during the previous outer iteration, allowing us to conclude that the option ``Adaptive-seq-dyn($c_0,c_h$)'' is the most preferable due to its efficiency and robustness.} 

\subsubsection{Computational results on other test instances} 
Finally, we present the performance of the best adaptive sequential SAA options (according to the above experiments on DEAK and DEAK-H instances) on an additional set of test instances that have a background in ``real-world'' applications. In particular, we consider Algorithm ``Adaptive-seq-BP-L($100$)'' and Algorithm ``Adaptive-seq-fixed(1.5)''. We consider Algorithm ``Adaptive-seq-fixed(1.5)'' rather than the one with dynamic rate, ``Adaptive-seq-dyn($1.05,3$)'', as we find in our experiments that the parameters $c_0$ and $c_h$ need to be fine tuned for specific instances in order to yield competitive performance. 

\begin{table}
\caption{Computational results of the adaptive partition-based level decomposition approach \cite{vanAckooij_Oliveira_Song_2016} (``PILD-ODA''), Algorithm \ref{algo:adaptive-sequential} with the fixed-width stopping criterion and sample size schedule proposed in \cite{bayraksan2012fixed} (``Adaptive-seq-BP-L($100$)'') and Algorithm \ref{algo:adaptive-sequential} with a geometrically increasing sample size schedule with rate $c_1 = 1.5$ (``Adaptive-seq-fixed(1.5)'') on an additional set of ``real-world'' test instances.}\label{table:5}
\begin{scriptsize}
\begin{tabular}{llllllllllll}
\hline
Ins & $N$ &  \multicolumn{2}{c}{PILD-ODA}  & & \multicolumn{3}{c}{Adaptive-seq-BP-L($100$)}  & & \multicolumn{3}{c}{Adaptive-seq-fixed($1.5$)}   \\
 \cline{3-4} \cline{6-8} \cline{10-12}
                           &     & Time       & M          &     & Time       & $M(L,n_L)$ & CI (cov.) &   & Time       & $M(L,n_L)$ & CI (cov.)    \\
                      \cline{3-12}
\multirow{3}{*}{LandS} & 50K   & 18.8 & 12 & & 0.2 & 10(2,364) & (0.1,100) & & 0.3 & 11(3,292) & (0.1,100) \\

 & 100K  & 35.6 & 12 & & 0.2 & 10(2,366) & (0.1,100) & & 0.3 & 11(3,298) & (0.1,100) \\

\hline
\multirow{3}{*}{gbd} & 50K   & 37.5 & 32 & & 0.5 & 24(3,602) & (0.0,89) & & 0.5 & 25(4,545) & (0.0,94) \\

 & 100K   & 75.9 & 29 & & 0.5 & 24(3,582) & (0.0,94) & & 0.5 & 25(4,576) & (0.0,94) \\
\hline
\multirow{3}{*}{cep} & 20K & 6.9 & 5 & & 0.1 & 4(1,280) & (0.0,99) & & 0.1 & 4(1,145) & (0.0,99) \\
& 50K & 17.1 & 4 & & 0.1 & 4(1,292) & (0.0,100) & & 0.1 & 4(2,151) & (0.0,100) \\

\hline
\multirow{3}{*}{pgp2} & 20K & 13.7 & 20 & & 2.1 & 65(4,700) & (0.1,66) & & 2.7 & 87(5,892) & (0.1,64) \\
& 50K & 31.0 & 22 & & 2.3 & 68(4,732) & (0.1,53) & & 2.2 & 77(4,727) & (0.1,49) \\
\hline
\multirow{3}{*}{4node} & 20K & 211.5 & 54 & & 2.2 & 64(1,146) & (0.0,80) & & 2.1 & 65(1,114) & (0.0,75) \\
& 50K & 487.4 & 51 & & 2.2 & 64(1,144) & (0.0,80) & & 1.9 & 63(1,111) & (0.0,82) \\
\hline

\multirow{3}{*}{retail} & 20K & 82.4 & 54 & & 140.0 & 503(16,3136) & (0.1,80) & & 87.0 & 305(10,6704) & (0.1,91) \\
& 50K & 179.3 & 53 & & 123.9 & 469(15,3040) & (0.1,78) & & 91.9 & 302(10,6998) & (0.1,86) \\
\hline
\multirow{3}{*}{baa99-20} & 20K & 735.3 & 187 & & 593.3 & 347(12,2346) & (0.1,98) & & 383.4 & 354(9,3454) & (0.1,100) \\
& 50K & 1670.4 & 184 & & 659.7 & 366(12,2344) & (0.1,100) & & 380.4 & 356(9,3349) & (0.1,100) \\
\hline
\multirow{3}{*}{20-term} & 2K & 1367.9 & 616 & & 2451.3 & 596(2,212) & (0.1,82) & & 1889.8 & 657(2,148) & (0.1,82) \\
& 5K & 1617.0 & 726 & & 2571.0 & 554(2,280) & (0.1,62) & & 2687.5 & 696(2,188) & (0.1,82) \\
\hline
ssn & 5K  & 6482.9 & 804 & & - & 2028(6,1104) & (17.0,100) & & - & 2477(7,1586) & (16.2,100) \\
\hline

\end{tabular}
\end{scriptsize}
\end{table}

From Table~\ref{table:5}, we see that our conclusions made based on the results from the DEAK instances also stand for most of this additional set of test instances, except instances ssn and 20-term, which we discuss separately since they serve as interesting negative examples. In particular, we see that both sequential sampling algorithms Adaptive-seq-BP-L($100$) and Adaptive-seq-fixed($1.5$) yield high-quality solutions and their solution quality validation much more efficiently than PILD-ODA in most cases. Using a geometric sequence for the sample size schedule (Adaptive-seq-fixed($1.5$) as opposed to Adaptive-seq-BP-L($100$)), further computational enhancements are obtained. The sequential sampling algorithms usually end up with a larger number of inner iterations than the deterministic algorithm PILD-ODA that employs the full set of samples. However, the computational savings brought by the smaller sample sizes used in the sequential sampling algorithms, which are reflected in the amount of work involved per inner iteration, turn out to offset the increase in the number of inner iterations on these instances. This is consistent with what our theoretical results presented in Section~\ref{sec:complexity}. In addition, we can observe some ``undercoverage'' phenomenon for pgp2 instances (as shown in column ``cov.''), which is somewhat expected as the variance associated with their solutions is quite large~\cite{bayraksan2006assessing}. Procedures that employ more than a single replication, such as A2RP proposed in~\cite{bayraksan2006assessing}, can be used to address the issue of ``undercoverage''.

As noted earlier, the problem instances ssn and 20-term are interesting as negative examples, where the proposed sequential sampling algorithms do not yield gains realized in other problem instances. Instance ssn is  challenging most probably due to the high inherent variance of the underlying random variables and the associated computational challenge in solving the second-stage problems while also reporting solution accuracy. For instance, observe from Table~\ref{table:5} that both options Adaptive-seq-BP-L($100$) and Adaptive-seq-fixed($1.5$) fail to provide confidence intervals with a satisfactory width within the stipulated time limit. We suspect that the variance associated with the second-stage optimal cost, along with the strict nature of the stopping criterion, contributes to ssn being in contrast with other test instances appearing in Table~\ref{table:5}. The negative effect of such high variance can be mitigated, at least in principle, by directly using variance reduction techniques, or through alternative stopping ideas such as that proposed in~\cite{sen2016mitigating}. 

The negative context presented by the instance 20-term appears to be different in spirit than ssn. Specifically, observe that Adaptive-seq-BP-L($100$) and Adaptive-seq-fixed($1.5$) exhibit longer computational times than the deterministic algorithm PILD-ODA on instances 20-term despite having a small number of outer iterations and small sample sizes used in each outer iteration. In fact, most of the computational effort is expended on solving the master problem, while the second-stage subproblems can be solved efficiently. The increased effort in solving the master problem could be because the ``warmstart" feature that worked well for other instances is not as effective here, since ``recovering'' a lower cutting-plane approximation using the dual vector information stored from previous iterations, although ``generated on the fly,'' requires the problem to be re-solved with a new right-hand-side at every re-start, and whenever any new first-stage decision vector is generated by the algorithm. This special feature of 20-term --- time-consuming master problems alongside easily solved second-stage problems --- means that our implementation's premise of the total computational burden being dominated by the task of solving second-stage LPs is not true in the 20-term context. The clear lesson from 20-term is then to adapt the implementation to explicitly account for the cost of solving the master problem alongside the cost of solving the second-stage problems, potentially leading to the use of a larger constant $c_1$ in such contexts. In addition, alternative ``warmstarting'' techniques for sequential sampling algorithms, such as those arising in stochastic decomposition~\cite{higle1991stochastic,higle2013stochastic} and stochastic dual dynamic programming algorithms~\cite{de2015improving}, may be more effective in relieving the computational challenges in repeatedly solving the master problem on these instances.

\section{CONCLUDING REMARKS}

We propose an adaptive sequential SAA algorithm to solve 2SLPs. During each iteration of the proposed framework, a piecewise linear convex optimization sample-path problem is generated with a scenario set having a specified size, and solved imprecisely to within a tolerance that is chosen to balance statistical and computational errors. We find that (i) the use of an appropriate solver to solve the sample-path problems, (ii) solving each sample-path problem only imprecisely to an appropriately chosen error tolerance, and (iii) the use of warm starts when solving sample-path problems, are crucial for efficiency.

Our theoretical results suggest that the optimality gap and the distance from the true solution set (of the generated stochastic iterates) converges to zero almost surely and in expectation. Moreover, when the sample sizes are increased according to a geometric rate, the fastest possible convergence rate under iid Monte Carlo sampling is preserved. This result is analogous to the $\mathcal{O}(\epsilon^{-2})$ optimal complexity rate for deterministic non-smooth convex optimization. Slower sample size increases result in a poorer convergence rate. Interestingly, the proposed framework also facilitates the use of dependent sampling schemes such as LHS, antithetic variates, and quasi-Monte Carlo without affecting convergence or the lower bound on the rate results. The use of such variance reduction ideas have been shown to be effective.

Our extensive numerical studies indicate that the proposed adaptive sequential SAA framework is able to produce high-quality solutions to 2SLPs significantly more efficiently than existing decomposition approaches that solve a single sample-path problem generated using a large sample size. Such gains are principally due to the sequential framework, the progressive increase in sample sizes in an optimal way, and the use of ``warm starts" in solving the sample-path problems. Our numerical experience has also revealed problem instances having certain challenging features that are not directly addressed by the implementations that we have used for illustration. These challenges could be mitigated by using alternative solvers that exploit particular problem structures and/or other termination criteria such as that proposed in~\cite{sen2016mitigating}. 

We believe that similarly efficient sequential SAA algorithms are possible for large-scale multi-stage convex stochastic programs, and possibly even stochastic integer programs. The key appears to be principled choices for adaptive sample sizes, solver for the sample-path problems, and adaptive optimality tolerance parameters. Ongoing research efforts are accordingly directed.

\section*{Acknowledgments} 
We greatly appreciate the comments and suggestions of the associate editor and two anonymous referees. The first author acknowledges support provided by the Office of Naval Research (ONR) through ONR Grant N000141712295, and by the National Science Foundation through the grant CMMI 1538050. The second author acknowledges partial support by the National Science Foundation (NSF) under grant CMMI 1854960. Any opinions, findings, and conclusions or recommendations expressed in this material are those of the authors and do not necessarily reflect the views of ONR or NSF.

\DeclareRobustCommand{\VAN}[3]{#3}
\DeclareRobustCommand{\DE}[3]{#3}

\bibliographystyle{siamplain}
\bibliography{adaptiveSeqSamp}

\end{document}